\documentclass[11pt,a4paper,reqno]{birkjour_t2}

\usepackage{amsmath,mathrsfs,ragged2e,multirow,array,booktabs}
\usepackage{color}
\usepackage{amsfonts}
\usepackage{amscd}
\usepackage{latexsym}
\usepackage{amssymb}  
\usepackage{amsthm}
\usepackage{hyperref}
\allowdisplaybreaks
\usepackage{graphicx, graphics}
\usepackage{subfig,float}
\usepackage[table]{xcolor}
\numberwithin{figure}{section}

\numberwithin{table}{section}
\numberwithin{equation}{section}
\newtheorem{theorem}{Theorem}[section]
\newtheorem{lemma}{Lemma}[section]

\newtheorem{remark}{Remark}[section]

\graphicspath{{Plots/Error_graph/},{Plots/Blowup_graph/},{Plots/Stability_graph/}}

\allowdisplaybreaks

\newcommand{\ust}{u^*}
\newcommand{\vst}{v^*}
\newcommand{\wst}{w^*}

\newcommand{\vo}{\bar{u}}

\newcommand{\ut}{u(\cdot,t)}
\newcommand{\vt}{v(\cdot,t)}
\newcommand{\wt}{w(\cdot,t)}

\newcommand{\us}{u(\cdot,s)}
\newcommand{\vs}{v(\cdot,s)}
\newcommand{\ws}{w(\cdot,s)}

\newcommand{\gu}{{\nabla u}}
\newcommand{\gv}{{\nabla v}}
\newcommand{\gvo}{{\nabla \vo}}
\newcommand{\gw}{{\nabla w}}

\newcommand{\lv}{{\Delta v}}
\newcommand{\lu}{{\Delta u}}

\newcommand{\mgu}{{|\gu|}}
\newcommand{\mgv}{{|\gv|}}
\newcommand{\mgw}{{|\gw|}}
\newcommand{\mgvo}{{|\gvo|}}

\newcommand{\nru}{\big\|u\big\|}

\newcommand{\nrw}{\big\|w\big\|}

\newcommand{\nut}{\big\|\ut\big\|}
\newcommand{\nvt}{\big\|\vt\big\|}
\newcommand{\nwt}{\big\|\wt\big\|}

\newcommand{\ngut}{\|\nabla\ut\|}

\newcommand{\ssup}{\sup\limits_{t\in(0,\tmax)}}
\newcommand{\tmax}{T_{\mathrm{max}}}
\newcommand{\tin}{t_0}
\newcommand{\into}{\int_{\Omega}}

\newcommand{\intt}{\int^t_0}
\newcommand{\intti}{\int^t_{\tin}}
\newcommand{\inti}{\int_0^\infty}
\newcommand{\ints}{\int_{\Omega}}
\newcommand{\ds}{\mathrm{d}s}
\newcommand{\dt}{\frac{\mathrm{d}}{\mathrm{d}t}}

\newcommand{\phvo}{\phi(\vo)}
\newcommand{\povo}{\phi^{'}(\vo)}
\newcommand{\ptvo}{\phi^{''}(\vo)}

\newcommand{\bup}{v^p}
\newcommand{\upo}{v^{p-1}}

\newcommand{\lis}{\mathcal{L}^{\infty}(\Omega)}
\newcommand{\los}{\mathcal{L}^{1}(\Omega)}
\newcommand{\lts}{\mathcal{L}^{2}(\Omega)}
\newcommand{\ltps}{\mathcal{L}^{\frac{2}{p}}(\Omega)}
\newcommand{\lps}{\mathcal{L}^{p}(\Omega)}
\newcommand{\lqs}{\mathcal{L}^{q}(\Omega)}
\newcommand{\lrs}{\mathcal{L}^{r}(\Omega)}
\newcommand{\wsin}{\mathcal{W}^{1,\infty}(\Omega)}
\newcommand{\wsq}{\mathcal{W}^{1,q}(\Omega)}
\newcommand{\cso}{\mathcal{C}^{0}(\overline{\Omega})}

\newcommand{\momeg}{|\Omega|}
\newcommand{\lros}{\mathcal{L}^{r\eta }(\Omega)}

\title[Existence, blow-up and numerical simulations for a fish-mussel chemotaxis system]{Global existence, blow-up behavior and numerical simulations for a class of chemotaxis-driven fish-mussel systems}


\begin{document}
	
%\linenumbers

\author[S. Gnanasekaran]{Gnanasekaran Shanmugasundaram}
\address[GS]{Department of Mathematics, National Institute of Technology Tiruchirappalli, Tamilnadu 620015, India}
\curraddr{}
\email{sekaran@nitt.edu}
\thanks{}

\author[J. Saha]{Jitraj Saha$^*$}
\address[JS]{Department of Mathematics, National Institute of Technology Durgapur, West Bengal 713209, India}
\curraddr{}
\email{jsaha.maths@nitdgp.ac.in}
\thanks{$^*$Corresponding author}
\begin{abstract}
	In this work, we investigate a chemotaxis-driven fish--mussel ecosystem model described by a coupled system of partial differential equations subject to homogeneous Neumann boundary conditions. Under suitable assumptions on the system parameters and initial data, we establish the global existence of classical solutions by employing semigroup methods together with a priori estimates. We also examine the possible blow-up behavior of solutions in a three-dimensional domain. To support the theoretical analysis, a finite element method is developed for the numerical approximation of the system, and convergence studies based on mesh refinement are carried out to verify the accuracy and stability of the proposed numerical scheme. Furthermore, numerical simulations illustrating the blow-up behavior of solutions in the computational domain are presented.
\end{abstract}

\maketitle

\section{Introduction and motivation}
\justifying
Aquaculture plays an important role in global food production and sustainable economic development, with ecosystem stability strongly influenced by nutrient availability, water quality, and species interactions. However, modern fish farming systems are increasingly affected by industrial discharge, agricultural runoff, and organic waste accumulation, which significantly alter nutrient dynamics and aquatic conditions \cite{gazi2009}. Excessive nutrient enrichment, particularly from nitrogen and phosphorus, can cause eutrophication, leading to algal blooms, oxygen depletion, water quality deterioration, and harmful impacts on fish populations and ecosystem stability. In addition, toxic algal blooms may transfer harmful substances through the food chain, posing ecological and human health risks. In contrast, oligotrophic water bodies are nutrient-deficient systems characterized by low algal productivity, high dissolved oxygen levels, and relatively clear and stable aquatic environments.

Early studies on aquaculture dynamics primarily relied on ordinary differential equation models to describe the temporal interactions between nutrients and aquatic populations. Subsequently, these models were extended to include additional species, such as fish and mussels, providing a more comprehensive representation of ecosystem processes. These studies demonstrated that nutrient availability and species interactions are key factors governing system stability and oscillatory behavior. An early contribution in this direction was the two-component nutrient--population model proposed by Ardito et al.~\cite{Ardito1988}. To address eutrophication problems commonly encountered in aquaculture systems, mussels were later incorporated into the modeling framework. Owing to their filter-feeding capacity, mussels remove suspended organic matter, improve water quality, and enhance oxygen availability within the ecosystem. Based on these ecological considerations, Gazi et al.~\cite{gazi2009} formulated a three-component fish farm model involving nutrients, fish, and mussels. The authors investigated the boundedness, local stability, and global stability of the proposed system. Furthermore, they analyzed the Hopf bifurcation of the model near the coexistence equilibrium point by considering the time delay as a bifurcation parameter. Numerical simulations were also carried out to validate the analytical findings. Their study demonstrated that the external food supply plays a crucial role in controlling the dynamics of the system. Subsequently, Arunkumar et al. \cite{MArun2024} analyzed the model using the Homotopy Analysis Method (HAM) and compared with existing Laplace–Adomian Decomposition Method (LADM) \cite{sambath2016} and the fourth-order Runge–Kutta method (RK4M). However, these studies were restricted to spatially homogeneous ordinary differential equation models and did not account for directed population movement induced by environmental gradients.

In realistic aquatic environments, fish populations tend to migrate toward nutrient-rich regions, leading to aggregation phenomena that can be effectively described through chemotaxis, namely the directed movement of organisms in response to chemical stimuli. Motivated by this limitation, we formulate a partial differential equation model incorporating chemotactic effects to investigate the spatiotemporal dynamics of nutrients, fish, and mussel populations. 
\begin{align}
	\left\{
	\begin{array}{ll}
		\tau u_{t} =\Delta u-u\big(\alpha +\beta v+\gamma w\big)+f(\cdot,t), \hspace{1cm}& \text{in}\:\: \Omega, t>0,\\
		v_{t}  =\Delta v-\chi\nabla\cdot\big(v\nabla u\big) +\sigma_1v\big(1- v+ u\big)- \delta_1 v, &\text{in}\:\: \Omega, t>0,\\
		\tau w_{t}  =\Delta w+\sigma_2uw-\delta_2w, &\text{in}\:\:\Omega, t>0,
	\end{array}
		\right.\label{1.1}
\end{align}
with initial condition
\begin{align}
	\tau u(\cdot, 0)=\tau u_{0}, \:v(\cdot, 0)=v_{0}, \: \tau w(\cdot, 0)=\tau w_{0}, \quad \text{in}\:\:\Omega,\label{1.2}
\end{align}
and homogeneous Neumann boundary condition
\begin{align}
	\frac{\partial u}{\partial\nu}=\frac{\partial v}{\partial\nu}=\frac{\partial w}{\partial\nu}=0,\quad \text{on}\quad\partial\Omega, t>0.\label{1.3}
\end{align}
Here, $\Omega\subset\mathbb{R}^n$ is a bounded domain with smooth boundary $\partial\Omega$, $\nu$ denotes unit outward normal vector on $\partial\Omega$, and $\tau\in\{0,1\}$. The functions $u(\cdot,t)$, $v(\cdot,t)$, and $w(\cdot,t)$ represent nutrient concentration, fish population density, and mussel population density, respectively. Positive constants $\alpha$, $\beta$, and $\gamma$ describe nutrient loss through natural outflow or sedimentation, nutrient uptake by fish, and nutrient uptake by mussels, respectively. The function $f(\cdot,t)$ represents an external nutrient input that may vary in both space and time. Trophic status of the aquatic ecosystem depends strongly on the magnitude of $f$: low nutrient input corresponds to an oligotrophic state, whereas excessive nutrient loading may induce eutrophication. Parameters $\sigma_1$ and $\delta_1$ denote intrinsic growth rate and mortality rate of the fish population, respectively. The chemotactic term $-\chi\nabla\cdot(v\nabla u)$ models directed movement of fish toward regions with higher nutrient concentration. The parameters $\sigma_2$ and $\delta_2$ denote the nutrient-to-biomass conversion efficiency and mortality rate of the mussel population, respectively. Through the consumption of particulate organic matter and phytoplankton generated by excess nutrients, mussels contribute to nutrient recycling and help regulate water quality. Consequently, the coupled interactions among nutrients, fish, and mussels promote ecological balance and support the long-term coexistence of species within the aquaculture system.

The initial conditions $u_0$, $v_0$ and $w_0$ are assumed to be nonnegative functions, ensuring the biological relevance of the model. Suppose that initial data satisfy
\begin{align}\label{1.4}
	\left\{
	\begin{array}{llll}
		&v_0\in\cso,\quad \mbox{with} \quad u_0 \geq 0\quad\mbox{in}\:\: \Omega,\\
		&\tau u_0, \tau w_0\in\wsq,\quad \mbox{for some}\,\, { q} >n,\quad \mbox{with} \quad \tau v_0, \tau w_0 \geq 0\quad\mbox{in}\:\: \Omega.
	\end{array}
	\right.
\end{align}
Moreover, the function $f(\cdot,t)$ satisfies
\begin{align}
f\in \mathcal{W}^{1, \infty}(\Omega\times (0, \infty)), \quad 	0\leq f\leq M,\:\: |\nabla f|\leq M,\quad \mbox{for some}\:\: M> 0.\label{1.5}
\end{align}
We briefly review existing results related to the present system. The classical Keller–Segel model provides a fundamental framework for describing chemotaxis and has been extensively studied, along with its variants, due to its significance in mathematics and biology. Recent advances on chemotaxis–consumption systems, including the existence, qualitative behavior, and pattern formation of solutions, are comprehensively surveyed in \cite{Lankit2023}.

It is worth noting that Keller and Segel \cite{keller} introduced the consumption model to describe the motion of {\it E.~coli} bacteria driven by oxygen gradients and nutrient consumption. The model is governed by the following system of partial differential equations
 \begin{align}
		 \left\{
		 \begin{array}{llll}
			u_t= \Delta u-\chi \nabla\cdot(u \gv),\\
			v_t= \Delta v- uv,
			 \end{array}
		 \right.\label{1.6}
\end{align} 
where $u$ denotes the bacterial density, $v$ represents the oxygen concentration, and $\chi>0$ is the chemotactic sensitivity coefficient. Model predicts the formation of traveling bacterial bands when the chemotactic sensitivity exceeds a critical threshold, a phenomenon that has also been observed experimentally. Tao \cite{ytao} established conditions ensuring the global existence of classical solutions to system \eqref{1.6} for spatial dimensions $n\geq 2$. Subsequently, Tao and Winkler \cite{taowinkler} investigated the existence of global weak solutions for $n\geq 3$ and further proved convergence of solutions toward spatially homogeneous equilibria as $t\to\infty$. Later, Zhang and Li \cite{zhangli} refined the criteria for global classical solvability in two-dimensional domains. They demonstrated that the condition proposed in \cite{ytao} remains essential for dimensions $n\geq 3$ and additionally established the exponential convergence of solutions as time progresses. Furthermore, Jiang et al. \cite{jiangwu} analyzed the blow-up behavior of system \eqref{1.6} in three spatial dimensions using a kinetic reformulation approach under the assumption that blow-up occurs. Their results generalized several previously known blow-up criteria and also provided insights into the local non-degeneracy of blow-up points. Baghaei and Khelghati \cite{kbaghaei} established the global existence of classical solutions to system \eqref{1.6} for various ranges of diffusion parameters. %A generalized form of system \eqref{1.6}, incorporating nonlinear diffusion and chemotactic sensitivity, is studied by Fan and Jin \cite{fan}. The chemotaxis system \eqref{1.7} with an additional nonlinear source term studied by Baghaei and Khelghati \cite{kbaghaei2}, \cite{xzhao}, Jiang and Han \cite{kjiang}

Fuest \cite{mfuest} investigated the existence and long-term behavior of solutions to the following chemotaxis system
 \begin{align*}%\label{1.9}
		 \left\{
		 \begin{array}{llll}
			 u_t= \Delta u-\chi \nabla\cdot(u \gv),\\
			 v_t= \Delta v- vw,\\
			 w_t= -\delta w+u.
			 \end{array}
		 \right.
\end{align*} 
In this model, additional variable $w$ describes the accumulation of a signaling or memory-related substance generated by the cell density $u$. It was shown that the system admits globally bounded solutions whenever either $n\leq 2$ or $\|v_0\|_{\lis}\leq \frac{1}{3n}$. Moreover, the solutions were proved to converge toward a spatially homogeneous equilibrium state as $t\to\infty$. Subsequently, Liu et al. \cite{yliu} extended these results by incorporating a logistic-type source term satisfying $f(u)\leq \mu(u-u^r)$ for dimensions $n\geq 2$.  Frassu and Viglialoro \cite{sfrassu} further refined and partially generalized the aforementioned results to higher-dimensional settings with $n\geq 5$. Subsequently, Chiyo et al. \cite{YChiyo2024} improved the results in \cite{sfrassu} and extended the analysis to the case where the chemotactic sensitivity $\chi$ depends on $v$. For further developments and related studies concerning this model, we refer the reader to \cite{jxiang}.

%Wang {\it et al.} \cite{jwang} investigated the global existence and asymptotic behavior of solutions to the following predator--prey system with indirect prey-taxis
% \begin{align}
%\left\{
% \begin{array}{llll}
% u_t=d_1 \Delta u- \nabla\cdot(\chi(w) u \nabla w)+bug(v)-uh(u),\\
% w_t= d_2\Delta w-\mu w+rv,\\
% v_t= d_3\Delta v+f(v)-ug(v).
% \end{array}
%\right.\label{1.10}
%\end{align} 
%where $u$ and $v$ denote the predator and prey densities, respectively, while $w$ represents the indirect signal mediating the taxis mechanism. The authors established the global existence of solutions and further derived precise convergence rates toward equilibrium states. In contrast to prey-taxis mechanisms, Ahn and Yoon \cite{iahn} studied a related system involving indirect predator-taxis. By constructing a suitable Lyapunov functional, they obtained stability results for spatial dimensions satisfying $n\leq 2$. Furthermore, in the presence of nonlinear diffusion, the existence of classical solutions was established in \cite{jxing2021} under appropriate assumptions on the model parameters. The authors also analyzed the asymptotic behavior of solutions under various predation mechanisms and interaction conditions.

Hu and Tao \cite{hutao} considered the following system
\begin{equation}\label{1.10}
	\left\{
	\begin{array}{llll}
		u_t= \Delta u-\chi \nabla\cdot(u \gv),\\
		\tau v_t= \Delta  v+\alpha w-\beta v-\gamma uv, \\
		\tau w_t= \Delta w-\delta uw+\mu w(1-w),
	\end{array}
	\right.
\end{equation}
with $\tau=\alpha=\beta=\gamma=\delta=\mu=1$. The authors investigated the well-posedness, solvability, and qualitative properties of classical solutions in two-dimensional domains. More precisely, they proved that, for any $\chi>0$ and sufficiently regular initial data, solutions remain globally bounded. Furthermore, when $\chi$ is sufficiently small and $\frac{1}{|\Omega|}\int_\Omega u_0<1$, convergence to equilibrium states as $t\to\infty$ was established, while the asymptotic behavior corresponding to $\frac{1}{|\Omega|}\int_\Omega u_0\geq 1$ was also discussed.

Subsequently, Gnanasekaran et al. \cite{gnanasekaran2025} studied the cases $\tau=\{0, 1\}$, focusing on the global existence and blow-up behavior of solutions in bounded domains of $\mathbb{R}^n$, $n\geq3$, under Neumann boundary conditions. More recently, Fuentes \cite{RDFuntes2025} studied the above system \eqref{1.10} with the source term $\mu_1u^{k}-\mu_2u^{k+1}$ in bounded domains of $\mathbb{R}^n$ $(n\geq 3)$ under Neumann boundary conditions. The author established global existence results by distinguishing the cases $k>1$ and $k=1$. For $k>1$, boundedness was obtained under conditions depending solely on the model parameters, whereas for $k=1$, an additional restriction involving $\chi$, $\mu_2$, $n$, and the initial datum $\|v_0\|_{L^\infty}$ was required.

Numerical simulations of the Keller–Segel chemotaxis system and its variants have been widely studied over the past decades. In the finite volume setting, Chertock and Kurganov \cite{AChertock2008} developed a second-order positivity-preserving central-upwind scheme and demonstrated its effectiveness for the classical Keller–Segel model, its extensions, and haptotaxis models. Subsequently, high-order hybrid finite-volume/finite-difference schemes for the Patlak–Keller–Segel system were proposed in \cite{AChertock2018, AKurganov2014}, providing improved accuracy while preserving key solution properties. More recently, Huang et al. \cite{XHuang2024} validated these schemes through extensive numerical experiments and investigated blow-up phenomena in three-dimensional Keller–Segel models in both parabolic–elliptic and parabolic–parabolic settings.

For Keller–Segel chemotaxis and related cancer invasion models, early numerical investigations focused on finite element discretizations and their convergence properties. Zhang et al. \cite{JZhang2016} proposed a splitting mixed finite element method for chemotaxis systems and established convergence and error estimates prior to blow-up. Subsequently, Ganesan et al. \cite{SGanesan2017} developed a Galerkin finite element framework for a cancer invasion model using Crank–Nicolson time integration, providing a robust baseline for simulating coupled reaction–diffusion–taxis dynamics. %Significant advances were made toward stability and positivity preservation in cancer invasion simulations. 
Manimaran et al. \cite{JManimaran2020} and Shangerganesh et al. \cite{LShangerganesh2020} introduced finite element schemes capable of producing nonnegative, oscillation-free solutions while accurately capturing steep gradients and blow-up behavior. Niño et al. \cite{VNino2021} developed two fully discrete finite element schemes for tumor invasion models that preserve the non-negativity of discrete variables. Khaled-Abad and Salehi \cite{LJKhaled2021} proposed a weak Galerkin finite element method for chemotaxis–haptotaxis systems using weak derivatives on discontinuous spaces, ensuring stable, nonnegative, and oscillation-free solutions even in the presence of cell-density blow-up. Further developments include the work of Pérez et al. \cite{JEPerez2022}, who extended finite element approaches for tumor invasion models with improved numerical robustness.

Recently, Aswin et al. \cite{VSAswin2023} investigated parallel space–time adaptive strategies for cancer invasion models involving density-dependent diffusion and haptotaxis. Their framework combines adaptive finite element discretization, adaptive time stepping, and dynamic load balancing via domain decomposition, significantly reducing computational cost while maintaining accuracy. Zhang et al. \cite{LZhang2025} developed a high-order explicit numerical method for haptotaxis models using compact finite differences and SSP Runge–Kutta time integration, achieving fourth-order spatial and third-order temporal accuracy while preserving positivity and improving numerical stability. More recently, Cengizci et al. \cite{SCengizci2026} proposed a stabilized SUPG finite element framework enhanced with a residual-based YZ$\gamma$ discontinuity-capturing operator and Crank–Nicolson time integration within FEniCS. Their method effectively eliminates nonphysical oscillations in convection-dominated regimes, providing a robust and efficient tool for simulating haptotaxis-driven cancer invasion models.

Motivated by the aforementioned studies, we investigate the analytical and numerical properties of the proposed chemotaxis system. We establish the global existence and boundedness of classical solutions and derive conditions characterizing potential blow-up phenomena. On the numerical side, we develop a finite element scheme and analyze its convergence and consistency using FEniCSx. Numerical experiments are presented to validate the theoretical findings and illustrate the predicted blow-up dynamics. We now state the main results of this work.

%\begin{theorem}\label{t1.1}
%Let $\Omega \subset\rsn (n\geq 3)$ be a  bounded domain with smooth boundary and $\tau \in \{0,1\}$. Suppose that the constants $\alpha, \beta, \gamma, \chi, \sigma_1, \sigma_2, \delta_1, \delta_2$ all are  positive and the function $f$ full fill \eqref{1.5}.  Then  for any nonnegative initial data $(\tau u_0, v_0, \tau w_0)$ satisfying \eqref{1.4} and additionally 
%	\begin{align*}
%\max\left\{\tau\|u_0\|_{\lis}, \frac{M}{\alpha}\right\}<\min\left\{\frac{\pi}{\chi}\sqrt{\frac{2}{n}}, \frac{\delta_2}{\sigma_2}\right\}
%	\end{align*}
%	system \eqref{1.1}-\eqref{1.3} admits a unique global classical solution $(u, v, w)$ which is uniformly bounded-in-time, in the sense that	
%	\begin{align*}
%		\nut_{\wsq}+\nvt_{\lis}+\nwt_{\wsq}\leq C, \quad\quad \text{for all}\quad t\in(0,\tmax),
%	\end{align*}
%	for some $C>0$
%\end{theorem}

\begin{theorem}\label{t1.1}
	Let $\Omega\subset\mathbb R^n$ $(n\ge3)$ be a bounded domain with smooth boundary, and let $\tau\in\{0,1\}$. Assume that
	$\alpha,\beta,\gamma,\chi,\sigma_1,\sigma_2,\delta_1,\delta_2$ all are  positive and that $f$ satisfies \eqref{1.5}. Then, for any nonnegative initial data
	$(\tau u_0,v_0,\tau w_0)$ satisfying \eqref{1.4} and
	\begin{align*}
	\max\left\{
	\tau\|u_0\|_{L^\infty(\Omega)},
	\frac{M}{\alpha}
	\right\}
	<
	\min\left\{
	\frac{\pi}{\chi}\sqrt{\frac{2}{n}},
	\frac{\delta_2}{\sigma_2}
	\right\},
\end{align*}
	the system \eqref{1.1}--\eqref{1.3} admits a unique global classical solution
	$(u,v,w)$ which is uniformly bounded in time. More precisely,
	there exists a constant $C>0$ such that
\begin{align*}
	\|u(\cdot,t)\|_{\wsq}
	+\|v(\cdot,t)\|_{\lis}
	+\|w(\cdot,t)\|_{\wsq}
	\le C,
	\qquad\text{for all }t\in(0,T_{\max}),
\end{align*}
	where $q>n$.
\end{theorem}

To derive an estimate for the blow-up time of solutions to system \eqref{1.1}-\eqref{1.3}, we introduce the following auxiliary functional.
\begin{equation}\label{t1.3.1}
	\Psi_\tau= \Psi_\tau(t):= \tau \into \mgu^4 + \into  v^2  + \tau \into \mgw^2 \quad \textrm{for all } t \in (0,\tmax),
\end{equation}
with initial value
\begin{equation}\label{t1.3.2}
	\Psi_\tau(0) = \tau \into |\nabla u_0|^4 + \into  v_0^2  + \tau \into \lvert \nabla w_0 \rvert^2.
\end{equation}
Indeed, while we have recalled that blow-up does not occur in two dimensions, the possibility of finite-time blow-up in higher dimensions remains an open problem.
\begin{theorem}\label{t1.3}
Suppose that $\Omega \subset\mathbb{R}^3$ is a bounded convex domain, $\tau\in\{0,1\}$ and the function $f$ satisfy \eqref{1.5}. Let $(u,v,w)$ be a solution of the system \eqref{1.1}-\eqref{1.3} corresponding to initial data satisfying \eqref{1.4}, blowing-up at finite time $\tmax$, in the sense that
\begin{align*}
\limsup_{t \to T_{max}} \lVert v(\cdot,t) \rVert_{\lis} = +\infty.
\end{align*}
Then, there exist computable constants $\mathcal{A}_\tau, \mathcal{B}_\tau, \mathcal{C}_\tau, \mathcal{D}_\tau$ such that
	\begin{align} 
		T_{\text{max}} \geq \int_{\Psi_\tau(0)}^{+\infty} \frac{\mathrm{d} \Psi_\tau(t)}{\mathcal{A}_\tau \Psi_\tau(t)^3 + \mathcal{B}_\tau \Psi_\tau(t)^\frac32 + \mathcal{C}_\tau \Psi_\tau(t)^\tau +\tau \mathcal{D}_\tau},\label{t1.3.3}
	\end{align}
where $\Psi_\tau$ is defined in \eqref{t1.3.1}. In particular, we can find a computable $\mathcal{E}_\tau>0$ such that this explicit lower bound
\begin{align}
	\tmax\geq \int_{\Psi_\tau(0)}^{+\infty} \frac{\mathrm{d}\Psi_\tau(t)}{\mathcal{E}_\tau\: \Psi_\tau(t)^3}=\frac{1}{2\mathcal{E}_\tau \:\Psi_\tau(0)^2}\label{t1.3.4}
\end{align}
can be derived.
\end{theorem}

The paper is organized as follows. Section \ref{sec2} presents the preliminary results and establishes the local existence of classical solutions. In Section \ref{sec3}, we prove the boundedness and global existence of classical solutions to system \eqref{1.1}-\eqref{1.5}. Section \ref{sec4} is devoted to the derivation of blow-up time estimates. Finally, Section \ref{sec5} contains numerical simulations that corroborate the theoretical findings on blow-up behavior.

\section{Preliminaries and local existence}\label{sec2}
We recall several useful inequalities and auxiliary lemmas that will be employed in the subsequent sections.
\begin{lemma}\label{l2.1}
	Let $y(t)$ be a positive absolutely continuous function on $(0, \infty)$ that satisfies
	\begin{align*}
		\left\{
		\begin{array}{rrll}
			\hspace*{-0.6cm}&&y'(t)+A y(t)^p\leq B,\\
			\hspace*{-0.6cm}&&y(0)=y_0,
			\end{array}
		\right.
		\end{align*}
	with some constants $A>0$, $B\geq 0$ and $p\geq 1$. Then we have 
	\begin{align*}
		y(t)\leq \max\left\{y_0, \: \left(\frac{B}{A}\right)^\frac{1}{p}\right\},
		\end{align*}
	for all $t\in(0,\tmax)$.
\end{lemma}
The following proof of the local existence lemma is adopted from well-established arguments proved in \cite{horstman}.

\begin{lemma}[Local Existence]\label{l2.2}
	Suppose that $\Omega\subset\mathbb{R}^n$, $n\geq 2,$ is a bounded domain with smooth boundary and $\tau \in \{0,1\}$. Moreover, assume that the function $f$ satisfies \eqref{1.5}. Then, for each nonnegative initial data satisfying \eqref{1.4}, there exists $\tmax\in (0,\infty]$ such that  the system \eqref{1.1}-\eqref{1.3} admits a unique nonnegative solution $(u, v, w)$ satisfying
	\begin{align*}
		v&\in { \mathcal{C}^{0}}\left(\overline{\Omega}\times\left.\left[0,\tmax\right.\right)\right)\cap { \mathcal{C}^{2,1}}\left(\overline{\Omega}\times\left(0,\tmax\right)\right),
\end{align*}
and
\begin{align*}
		u, w&\in { \mathcal{C}^{0}}\left(\overline{\Omega}\times\left.\left[0,\tmax\right.\right)\right)\cap { \mathcal{C}^{2,1}}\left(\overline{\Omega}\times\left(0,\tmax\right)\right)\cap { \mathcal{L}^{\infty}_{loc}}\left(\left.\left[0,\tmax\right.\right);\wsq\right),  \quad \mbox{for}\:\: \tau=1
	\end{align*}
or
	\begin{align*}
	u, w&\in {\mathcal{C}^{2,0}}\left(\overline{\Omega}\right)\cap { \mathcal{L}^{\infty}_{loc}}\left(\left.\left[0,\tmax\right.\right);\wsq\right),  \quad \mbox{for}\:\: \tau=0
\end{align*}
	classically in $\Omega\times (0,\tmax)$. Furthermore, if $\tmax<\infty$, then
	\begin{equation} \begin{split}
			\lim_{t\to \tmax}\Big(\nut_{\wsq}+\nvt_{\lis}+\nwt_{\wsq}\Big)= \infty.\label{l1.1}
	\end{split} \end{equation} 
\end{lemma}
\begin{proof}
The proof follows from standard arguments based on the Banach fixed point theorem together with elliptic and parabolic regularity theory. For details, we refer to \cite{horstman}. Moreover, the non-negativity of the solution in $\Omega\times(0,\tmax)$ follows from the maximum principle in conjunction with the initial data \eqref{1.4}.
\end{proof}
Throughout the sequel, $(u,v,w)$ denotes the local-in-time solution of system \eqref{1.1}-\eqref{1.3} corresponding to the initial data \eqref{1.4}, as provided by Lemma \ref{l2.2}.

\begin{lemma}\label{l2.3}
	The classical solution $(u,v,w)$ satisfies
	\begin{align}
		 \nut_{\lis}&\leq K_1:=\max\left\{\tau\|u_0\|_{\lis}, \frac{M}{\alpha}\right\}, \quad\text{for all}\:\: t\in(0, \tmax),\label{l2.3.1}\\
		\nvt_{\los}&\leq K_3:=\max\Big\{\|v_0\|_{\los}, (1+K_1) \momeg\Big\},\quad \text{for all}\:\: t\in(0, \tmax).\label{l2.3.3}
	\end{align} 
	Moreover, if $K_1\leq \frac{\delta_2}{\sigma_2}$,
	\begin{align}
	\nwt_{\lis}\leq K_2:= \|w_0\|_{\lis}, \quad\text{for all}\:\: t\in(0, \tmax).\label{l2.3.2}
	\end{align} 
	\begin{proof}
	Estimate \eqref{l2.3.1} follows from the elliptic and parabolic
	maximum principles. Moreover, if
	\(K_1\le \delta_2/\sigma_2\), then
	\(\sigma_2u-\delta_2\le0\), and the maximum principle applied to the
	third equation in \eqref{1.1} yields \eqref{l2.3.2} (see \cite[$\S$6.4 and $\S$7.1.4]{lcevans}). Next, integration of the second equation in \eqref{1.1} gives
	\begin{align*}
	\dt\ints v\leq\sigma_1\left(1+\nru_{\lis}\right)\ints v -\sigma_1 \ints v^2\quad \text{for all $t\in(0,\tmax)$}.
	\end{align*}
	Using Cauchy-Schwarz inequality,
	\begin{align*}
	\dt\ints v\leq\sigma_1\left(1+\nru_{\lis}\right)\ints v -\frac{\sigma_1}{\momeg} \left(\ints v\right)^2\quad \text{for all $t\in(0,\tmax)$}.
	\end{align*}
	Applying Lemma \ref{l2.1} together with \eqref{l2.3.1} yields \eqref{l2.3.3}. More specifically, when $\tau=0$, $\underline{u}\equiv0$ is a subsolution of the first equation in \eqref{1.1}. This implies $u\geq0$ and $ \frac{M}{\alpha}$ is the supersolution for all $(x,t)\in\Bar{\Omega}\times(0,\tmax)$. On the other hand, the third equation shows that $\underline{w}\equiv0$ is a subsolution, so that $w\geq 0$ for all $(x,t)\in\Bar{\Omega}\times(0,\tmax)$. Using these observations,  $\underline v\equiv0$ is the subsolution for the second equation, yielding $v\ge0$. The case $\tau=1$ is analogous, with the initial conditions
	taken into account in the comparison argument.
	\end{proof}
\end{lemma}

\begin{lemma}[Extensibility criterion]\label{l2.4}
	Suppose there exists $\displaystyle{p>\frac{n}{2}\geq 1}$ and let $q>n$ such that
	\begin{align}
		\ssup \nvt_{\lps}<\infty.\label{l5.1}
	\end{align}
	Then we have
	\begin{equation*} 
		\ssup\Big(\nut_{\wsq}+\nvt_{\lis}+\nwt_{\wsq}\Big)< \infty. %\label{l5.2}
	\end{equation*} 
	In particular $\tmax=\infty$ and $v\in \mathcal{L}^\infty ((0,\infty),\mathcal{L}^\infty(\Omega))$.
\end{lemma}
\begin{proof}
	Let $q>n$ and for each fixed $p>\frac{n}{2}$, there hold
	\begin{align*} %\label{l5.3}
		\frac{np}{(n-p)_+}=
		\left\{
		\begin{array}{llll}
			\infty,  &\quad \mbox{if}\quad p\geq n,  \\
			\frac{np}{n-p}>n, & \quad \mbox{if}\quad \frac{n}{2}< p < n,
		\end{array}
		\right.
	\end{align*} 
	and choose $q<\frac{np}{(n-p)_+}$ and $1<r<q$ fulfilling $n<r<\frac{np}{(n-p)_+}$
	which enables to choose $\eta>1$, \: $n < r\eta<\frac{np}{(n-p)_+}$ and $ r \eta \leq q$. We fix arbitrary $t\in(0,\: \tmax)$.
	We distinguish between the cases $\tau=1$ and $\tau=0$.\\
	\textbf{Case $\tau=1$.}
	Applying the variation of constants formula to the first equation of $\eqref{1.1}$, we get
	\begin{align*}
		u=e^{-\alpha t}e^{t\Delta} u_0-\intt e^{-\alpha(t-s)} e^{(t-s)\Delta} \us\Big(\beta \vs+\gamma \ws\Big)\ds.%\label{l5.4}
	\end{align*}
	From the above, we have the estimate
	\begin{align*}
		\ngut_{\lros}\leq&\: e^{-\alpha t}\|\nabla e^{t\Delta}u_0\|_{\lros}\\
		&+\nut_{\lis} \intt e^{-\alpha(t-s)}\Big\|\nabla e^{(t-s)\Delta}\left(\beta \vs+\gamma \ws\right)\Big\|_{\lros}\ds\\
		&+\intt e^{-\alpha(t-s)} \Big\|\nabla e^{(t-s)\Delta} f(\cdot,s)\Big\|_{\lros}\ds.
	\end{align*}
	By using the estimates for the Neumann heat semigroup (Winkler \cite{winkler}) and $r\eta \leq q$, we obtain
	\begin{align*}
		\ngut_{\lros} \leq&\: c_1e^{-\alpha t}\|u_0\|_{\wsq} \\
		&+c_2\intt e^{-\alpha(t-s)}\left(1+\left(t-s\right)^{-\frac{1}{2}-\frac{n}{2}\left(\frac{1}{p}-\frac{1}{r\eta}\right)}\right)e^{-\lambda(t-s)} \Big\|\vs\Big\|_{\lps}\ds \\
		&+c_2\intt e^{-\alpha(t-s)}\left(1+\left(t-s\right)^{-\frac{1}{2}-\frac{n}{2}\left(\frac{1}{p}-\frac{1}{r\eta}\right)}\right)e^{-\lambda(t-s)} \Big\|\ws\Big\|_{\lps}\ds \\
		&+c_2\intt e^{-\alpha(t-s)}\left(1+\left(t-s\right)^{-\frac{1}{2}-\frac{n}{2}\left(\frac{1}{p}-\frac{1}{r\eta}\right)}\right)e^{-\lambda(t-s)} \Big\|f(\cdot,s)\Big\|_{\lps}\ds \\
		\leq&\: c_1\|u_0\|_{\wsq}\\
		& + c_2\sup_{t\in (0, \tmax)}\nvt_{\lps}\intt e^{-\alpha(t-s)}\left(1+\left(t-s\right)^{-\frac{1}{2}-\frac{n}{2}\left(\frac{1}{p}-\frac{1}{r\eta}\right)}\right)e^{-\lambda(t-s)}\ds \\
		&+ c_2\nwt_{\lis}\intt e^{-\alpha(t-s)}\left(1+\left(t-s\right)^{-\frac{1}{2}-\frac{n}{2}\left(\frac{1}{p}-\frac{1}{r\eta}\right)}\right)e^{-\lambda(t-s)} \ds\\
		&+ c_2M\intt e^{-\alpha(t-s)}\left(1+\left(t-s\right)^{-\frac{1}{2}-\frac{n}{2}\left(\frac{1}{p}-\frac{1}{r\eta}\right)}\right)e^{-\lambda(t-s)} \ds,
		\end{align*}
	where $c_1$ and $c_2$ are positive constants. Because of our assumptions $r\eta<\frac{np}{(n-p)}$ and $r \eta \leq q$, we can ensure that $\frac{1}{2}+\frac{n}{2}\left(\frac{1}{p}-\frac{1}{r\eta}\right) <1$.
	In addition, the Gamma function gives \\
	%	\begin{align*}
		%		\inti x^{-n}  e^{-\lambda x}  = \lambda^{n-1}\: \Gamma (1-n), \qquad \mathrm{for}\quad Re(n)<1, Re(\lambda)>0,
		%	\end{align*}
	$\inti e^{-\alpha\varphi}\left(1+\varphi^{-\frac{1}{2}-\frac{n}{2}(\frac{1}{p}-\frac{1}{r\eta})}\right)e^{-\lambda\varphi} < \infty$. 
	Using Lemma \ref{l2.3}, we can conclude that
	\begin{equation} 
		\ngut_{\lros}\leq c_4, \qquad \forall\, t\in (0, \tmax),\label{l5.6}
	\end{equation} 
	where $c_4>0$. Next let $\tin =\mathrm{max}\{0,\: t-1\}$ and use the variation of constants formula to the second equation of $\eqref{1.1}$, to get
	\begin{align}
			v(\cdot, t)=&e^{-\delta_1(t-\tin)}e^{(t-\tin)\Delta} v(\cdot, \tin)-\chi\intti e^{-\delta_1(t-\tin)} e^{(t-s)\Delta}\nabla\cdot\Big(\vs\nabla\us\Big)\ds\nonumber\\
			&+\sigma_1\intti e^{-\delta_1(t-\tin)}e^{(t-\tin)\Delta} \vs\Big(1-\vs+\us\Big)\ds\label{l5.7}.
			%&+\mu_1\intti e^{(t-s)\Delta}\buks\Big(1-\bus\Big)\ds.
	\end{align} 
	Next, taking $\lis$ norm on both sides of \eqref{l5.7}, we obtain
	\begin{align*}
		\nvt_{\lis}\leq &\: e^{-\delta_1(t-\tin)}\big\|e^{(t-\tin)\Delta} v(\cdot, \tin)\big\|_{\lis}\\
		&+\chi\intti e^{-\delta_1(t-\tin)}\Big\|\nabla e^{(t-s)\Delta}\Big(\vs\nabla \us\Big)\Big\|_{\lis}\ds\\
		&+\sigma_1\intti e^{-\delta_1(t-\tin)}\left\|e^{(t-\tin)\Delta} \vs\Big(1-\vs+\us\Big)\right\|_{\lis}\ds,
	\end{align*}
	for all $t\in(0,\tmax)$. If $t\leq 1$, then $\tin=0$ and we can use the maximum principle to get
	\begin{align*}
		\big\|e^{(t-\tin)\Delta} v(\cdot,\tin)\big\|_{\lis}= \big\|e^{(t-s_0) \Delta} v(\cdot, 0)\big\|_{\lis} \leq \big\| v(\cdot, 0)\big\|_{\lis}.
	\end{align*}
	If $t>1$, again using the Neumann heat semigroup (Winkler \cite{winkler}) property, with $c_5>0$,
	\begin{align*}
		\big\|e^{(t-\tin)\Delta} v(\cdot,\tin)\big\|_{\lis}\leq c_5(t-\tin)^{-\frac{n}{2}}\big\| v(\cdot,\tin)\big\|_{\los} \leq c_5 K_3,
	\end{align*}
	because $t-\tin=1$. Using Cauchy's inequality
\begin{align*}
\sigma_1 v(1-v+u)= \sigma_1 v(1+u)-\sigma_1 v^2 \leq \sigma_1  v^2 +\frac{\sigma_1(1+u)^2}{4}-\sigma_1 v^2\leq \frac{\sigma_1(1+u)^2}{4}.
	\end{align*}
	In view of the estimates for the Neumann heat semigroup (Winkler \cite{winkler}),  
	$c_6 > 0$ satisfies
		\begin{equation} \begin{split}
				\nvt_{\lis}\leq &\:\max\Big\{\| v(\cdot, 0)\|_{\lis}, c_5 K_3\Big\}  \\
				&+c_6\intti\left(1+\left(t-s\right)^{-\frac{1}{2}-\frac{n}{2}\left(\frac{1}{r}-\frac{1}{\infty}\right)}\right)e^{-\lambda(t-s)}\Big\|\vs\nabla \us\Big\|_{\lrs}\ds\\
				&+\frac{\sigma_1(1+K_1)^2}{4}\intti\left(1+\left(t-s\right)^{-\frac{1}{2}-\frac{n}{2}\left(\frac{1}{r}-\frac{1}{\infty}\right)}\right)e^{-\lambda(t-s)}\ds.\label{l5.8}
		\end{split} \end{equation} 
	%where $\frac{1}{2}+\frac{n}{2q_0}<1$ because of $q_0>n$. %Once again using the Gamma function gives\\ $\inti\left(1+\psi^{-\frac{1}{2}-\frac{n}{2q_0}}\right)e^{-\lambda\psi} < \infty.$ %Now, using the Youngs's inequality, we have
	%\begin{align*}
	%\mu_1\buk\big(1-\bu\big)%\leq &\:  \mu_1\buk-\mu_1\bu^{\kappa+1}\\
	%\leq&\: \mu_1 \bu^{\kappa+1}+\mu_1\left(\frac{\kappa+1}{\kappa}\right)^{-\kappa}\frac{1}{\kappa+1}-\mu_1\bu^{\kappa+1}
	%\leq c_7.
	%\end{align*}
	%The last term in \eqref{l5.8} can be written as
	%\begin{equation} \begin{split}
	%\mu_1\intti\left\|e^{(t-s)\Delta}\buks\left(1-\bus\right)\right\|_{\lis}\ds\leq &\:  \mu_1\intti\Big\|\buks\Big(1-\bus\Big)\Big\|_{\lis}\ds  \\
	%&\leq  c_7[t-\tin].\label{l2.9}
	%\end{split} \end{equation} 
	%If $t\leq 1$, then $\tin=0$ and we have $c_7[t-t_0]\leq c_8$ and if $t>1$,
	%$t-t_0=1$, then we have $c_7[t-s_0]=c_7$. 
	%Therefore, inserting \eqref{l2.9} into \eqref{l5.8}, one gets
	%\begin{equation} \begin{split}
	%\nut_{\lis}\leq & \:\max\Big\{\|\bu(\cdot, 0)\|_{\lis}, c_5M_1\Big\}  \\
	%&+ c_6\intti\left(1+\left(t-s\right)^{-\frac{1}{2}-\frac{n}{2}\left(\frac{1}{q_0}-\frac{1}{\infty}\right)}\right)e^{-\lambda(t-s)}\Big\|\bus\gvs\Big\|_{\lrs}\ds  \\
	%&+c_8,\label{l5.10}
	%\end{split} \end{equation} 
	%for all $t\in(0,\tmax)$. 
	Here by using the H\"older inequality and the Interpolation inequality and also
	using \eqref{l2.3.3} and \eqref{l5.6}, we obtain
	\begin{equation} \label{l5.11}
		\begin{split}
			\big\|\vs\nabla \us\big\|_{\lrs} \leq &\:  \:\big\|\vs\big\|_{{ L^{r\widehat{\eta} }}(\Omega)}\: \big\|\nabla \us\big\|_{\lros} \\
			\leq &\:  \: \big\|\vs\big\|_{\lis}^{\theta}\: \big\|\vs\big\|_{\los}^{1-\theta}\: \big\|\nabla \us\big\|_{\lros} \\
			\leq &\:  c_{7} \big\|\vs\big\|_{\lis}^{\theta},
		\end{split}
	\end{equation}
	where $\widehat{\eta}$ is the dual exponent of $\eta$ and $\theta=1-\frac{1}{r \widehat{\eta}} \in (0,1)$, for all $s\in (\tin, t)$ and $c_{7}>0$.  Inserting \eqref{l5.11} in \eqref{l5.8} and using the Gamma function, we obtain  $\inti\left(1+\varphi^{-\frac{1}{2}-\frac{n}{2r}}\right)e^{-\lambda\varphi} < \infty$, where $\frac{1}{2}+\frac{n}{2r}<1$ because of $r>n$. This gives
	\begin{equation*} %\label{l5.12}
		\nvt_{\lis}\leq \:\max\Big\{\|v(\cdot, 0)\|_{\lis}, c_5 K_3\Big\}+c_{8} \big\|\vs\big\|_{\lis}^{\theta}.
	\end{equation*} 
	Finally, using the Young inequality, we obtain
	\begin{equation} \begin{split}
			\nvt_{\lis}\leq &c_{9}, \hspace*{1cm} \forall t\in(0, \tmax),\label{l5.13}
	\end{split} \end{equation} 
	where $c_{9}>0$.
	Similarly applying the variation of constants formula to the third equation of $\eqref{1.1}$, we get
	%\begin{align*}
	%	\bwt=e^{t\Delta}\bw_0-\delta\intt e^{(t-s)\Delta}\buls\bws\ds+\mu_2\intt e^{(t-s)\Delta}\bwms\big(1-\bws\big)\ds.%\label{l5.4}
	%\end{align*}
	%Now,
	\begin{align*}
		\|\nabla\wt\|_{\lros}\leq&\:e^{-\delta_2(t-\tin)}\|\nabla e^{t\Delta} w_0\|_{\lros}+\sigma_2 \intt e^{-\delta_2(t-\tin)}\Big\|\nabla e^{(t-s)\Delta}\us\ws\Big\|_{\lros}\ds.
	\end{align*}
	By using the estimates for the Neumann heat semigroup (Winkler \cite{winkler}), we obtain
	\begin{align*}
		\|\nabla\wt\|_{\lros} \leq&\: c_{10}\|w_0\|_{\wsq}  \\
		&+c_{11}\nut_{\lis}\intt\left(1+\left(t-s\right)^{-\frac{1}{2}-\frac{n}{2}\left(\frac{1}{p}-\frac{1}{r\eta}\right)}\right)e^{-\lambda(t-s)}\Big\|\ws\Big\|_{\lps}\ds.
	\end{align*}
	%where $c_1$ and $c_2$ are positive constants. Because of our assumption $q_0 \zeta<\frac{np}{(n-p)}$ and $q_0 \zeta \leq q$, we can ensure that $\frac{1}{2}+\frac{n}{2}\left(\frac{1}{p}-\frac{1}{q_0\zeta}\right) <1$,
	%\begin{align*}
	%	\inti x^{-n}  e^{-\lambda x}  = \lambda^{n-1}\: \Gamma (1-n), \qquad \mathrm{for}\quad Re(n)<1, Re(\lambda)>0,
	%\end{align*}
	%$\inti e^{-\alpha\psi}\left(1+\psi^{-\frac{1}{2}-\frac{n}{2}(\frac{1}{p}-\frac{1}{q_0\zeta})}\right)e^{-\lambda\psi} < \infty$.
	Using the Gamma function and previous results, finally  we obtain
	\begin{align*}
			\|\nabla\wt\|_{\lros}&\leq c_{13}, \qquad t\in(0, \tmax),
	\end{align*} 
	where $c_{13}>0$. Since the initial data satisfy $u_0,w_0\in\wsq$, the smoothing property of the heat semigroup implies that $u,w\in \mathcal{L}^\infty\bigl((0,\tmax);\wsq\bigr)$. Consequently,
	\begin{align}
	\ssup\Bigl(\nut_{\wsq}+\nvt_{\lis}+\nwt_{\wsq}\Bigr)<\infty.\label{l2.4.6}
	\end{align}
\textbf{Case $\tau=0$.}
In this case, the first and third equations in \eqref{1.1} are
elliptic. Hence, the variation-of-constants arguments used above for
$u$ and $w$ are no longer applicable. We therefore treat these two
components by elliptic regularity. Since $f\in \wsin$, we have
$f\in \lis\subset \lps$. Hence, by
\eqref{l2.3.1} and \eqref{l2.3.2}, for every
$t\in(0,T_{\max})$, the first equation in \eqref{1.1} with $\tau=0$ yields
\begin{align*}
	\|\Delta u(\cdot,t)\|_{\lps}
	&\le
	\|u(\cdot,t)\|_{\lis}
	\Big(
	\alpha |\Omega|^{1/p}
	+\beta\|v(\cdot,t)\|_{\lps}
	+\gamma\|w(\cdot,t)\|_{\lps}
	\Big)
	+\|f(\cdot,t)\|_{\lps}\le c_{14}.
\end{align*}
Consequently, elliptic $\lps$-regularity for the homogeneous Neumann
problem gives
\begin{align*}
	\|u(\cdot,t)\|_{\mathcal{W}^{2,p}(\Omega)}
	&\leq c_{15}\left(
	\|\Delta u(\cdot,t)\|_{\lps}
	+\|u(\cdot,t)\|_{\lps}
	\right).
\end{align*}
Hence,
\begin{align*}
	\sup_{0<t<T_{\max}}
	\|u(\cdot,t)\|_{\mathcal{W}^{2,p}(\Omega)}
	&\le c_{16}.
\end{align*}
%Here we have also used the uniform $L^\infty$-bound for $u$ to control the zeroth-order term in the elliptic estimate.
Since $p>\frac n2$, we can choose $q>n$ such that $\mathcal{W}^{2,p}(\Omega)\hookrightarrow \wsq$.
Indeed, if $p<n$, it is sufficient to choose $n<q<\frac{np}{n-p}$,
while for $p\ge n$ the embedding into $\wsq$ holds for every
finite $q>n$. Thus,
\begin{equation}
	\sup_{0<t<T_{\max}}
	\|u(\cdot,t)\|_{\wsq}
	\le c_{17}.\label{l2.4.7}
\end{equation}
We next establish the boundedness of $v$.
Form \eqref{l2.3.1}, we have
\begin{align*}
	\sigma_1v(1-v+u)-\delta_1v
	\le
	\sigma_1v(1+u)-\sigma_1v^2\le
	\frac{\sigma_1}{4}(1+u)^2
	\le c_{18},
\end{align*}
Since $q>n$, \eqref{l2.4.7} in particular implies
\begin{equation}\label{l2.4.8}
	\sup_{0<t<T_{\max}}
	\|\nabla u(\cdot,t)\|_{\lqs}
	\leq c_{19}.
\end{equation}
By the standard regularity theory for
parabolic convection-diffusion equations together with
\eqref{l5.1} and \eqref{l2.4.8}, we obtain
\begin{equation}
	\sup_{0<t<T_{\max}}
	\|v(\cdot,t)\|_{\lis}
	\le c_{20}.\label{l2.4.9}
\end{equation}
Finally, since $\tau=0$, the third equation in \eqref{1.1} is also
elliptic and by \eqref{l2.3.1} and \eqref{l2.3.2}, we have
\[
\sup_{0<t<T_{\max}}
\|(\delta_2-\sigma_2u(\cdot,t))w(\cdot,t)\|_{\lqs}
\leq c_{21}.
\]
Hence, by elliptic $\mathcal{L}^q$-regularity for the homogeneous Neumann
problem,
\begin{equation*}
	\sup_{0<t<T_{\max}}
	\|w(\cdot,t)\|_{\mathcal{W}^{2,q}(\Omega)}
	\leq c_{22}.
\end{equation*}
In particular,
\begin{equation}
	\sup_{0<t<T_{\max}}
	\|w(\cdot,t)\|_{\wsq}
	\leq c_{23}.
\end{equation}
Combining \eqref{l2.4.7}, \eqref{l2.4.8}, and
\eqref{l2.4.9}, we obtain
\begin{equation}
	\sup_{0<t<T_{\max}}
	\left(
	\|u(\cdot,t)\|_{\wsq}
	+
	\|v(\cdot,t)\|_{\lis}
	+
	\|w(\cdot,t)\|_{\wsq}
	\right)
	\le c_{24},
	\qquad q>n.\label{l2.4.11}
\end{equation}
This completes the proof for the case $\tau=0$. These uniform bounds \eqref{l2.4.6} and \eqref{l2.4.11} are incompatible with the extensibility criterion
for the local classical solution at a finite maximal time. Consequently, $T_{\max}=\infty$. In particular, $
v\in \mathcal{L}^\infty\big((0,\infty); \lis\big)$.
%Assume, by contradiction that $\tmax<\infty$. Then the above estimate contradicts the extensibility criterion \eqref{l1.1}. Therefore, $\tmax=\infty$, and the proof is complete.
\end{proof}
This result provides a boundedness criterion that will be instrumental in establishing the boundedness of solutions to the model under consideration. Furthermore, it plays a key role in showing that any unbounded solution must be accompanied by blow-up of an appropriate energy functional. More precisely, we have the following result.

%\begin{remark}\label{r2.1} For $n=3$, if $(u,v,w)$ is a solution to model \eqref{1.1}-\eqref{1.3} which blows-up in finite time $\tmax$ in the sense that 
%	\begin{equation*} 
%		\limsup_{t\to \tmax} \nvt_{\lis} = \infty,
%	\end{equation*} 
%	then $\lim_{t \to T_{max}} \Psi_\tau(t) = +\infty$, where $\Psi_\tau$ is defined in \eqref{t1.3.1}.
%	Indeed, if $\into v^2$ were uniformly bounded-in-time on $(0,\tmax)$, necessarily in view of Lemma \ref{l2.4} we would obtain that $v\in L^\infty((0,\infty);L^\infty(\Omega))$, which is a contradiction. 
%\end{remark}

\begin{remark}\label{r2.1}
	Let $n=3$. If $(u,v,w)$ is a solution of \eqref{1.1}--\eqref{1.3} that blows up at the finite time $\tmax$ in the sense that $
	\limsup_{t\to\tmax}\|v(\cdot,t)\|_{\lis}=\infty$,
	then $\lim_{t\to\tmax}\Psi_\tau(t)=+\infty$, where $\Psi_\tau$ is defined in \eqref{t1.3.1}. Indeed, if $\int_\Omega v^2(\cdot,t)\,dx$ were uniformly bounded for $t\in(0,\tmax)$, then Lemma \ref{l2.4} would imply that $v\in \mathcal{L}^\infty\bigl((0,\tmax); \mathcal{L}^\infty(\Omega)\bigr)$,
	which contradicts the blow-up assumption.
\end{remark}

%\subsection{Some general tools} \label{subsec:GenTools}
%We will rely on the following general results.

The next lemma (see the main ideas in \cite{qzhang} and \cite{kbaghaei2}, inspired by \cite{ytao}) will be used to prove Theorem \ref{t1.1} in the case $\tau=1$.

%\begin{lemma}\label{l2.5}
%	Let $\epsilon\in(0,1)$ and $p>1$. Define the function
%	\begin{align*}
%		\phi(x):=e^{\zeta(x)}, \quad 0\leq x\leq N,
%	\end{align*}
%	where
%	\begin{align*}
%		\zeta(x):=-\frac{l}{2m}x+\frac{\sqrt{4km-l^2}}{2m}\int^x_0\tan\left(\frac{\sqrt{4km-l^2}}{2r}s+\arctan\frac{l}{\sqrt{4km-l^2}}\right)\ds
%	\end{align*}
%	with $k=(p-1)^2, l=-4(p-1)\epsilon, m=\frac{4}{p}(1+(p-1)\epsilon)$ and $r=\frac{4}{p}(p-1)(1-\epsilon)$. If
%	\begin{equation} \begin{split}
%			N<\frac{2}{\sqrt{p}}\sqrt{\frac{1-\epsilon}{1+p\epsilon}}\left(\frac{\pi}{2}+\arctan\sqrt{\frac{p}{1+(p-1)\epsilon-p\epsilon^2}}\epsilon\right),\label{l6.1}
%	\end{split} \end{equation} 
%	then the function $\phi(x)$ is well defined and satisfies the following conditions
%	\begin{equation*} %\label{l6.2}
%		1\leq \phi(x)\leq \phi(N), \qquad 0\leq \phi^{'}(x)<\infty
%	\end{equation*} 
%	and
%	\begin{equation} \begin{split}
%			\frac{1}{p}\phi^{''}(x)-\phi^{'}(x)\geq 0.\label{l6.3}
%	\end{split} \end{equation} 
%	Moreover we have
%	\begin{equation} \begin{split}
%			|(p-1)\phi(x)-2\phi^{'}(x)|-2\sqrt{(p-1)(1-\epsilon)\phi(x)\left(\frac{1}{p}\phi^{''}(x)-\phi^{'}(x)\right)}=0\label{l6.4}
%	\end{split} \end{equation} 
%	for all $0\leq x\leq N$.
%\end{lemma}

\begin{lemma}\label{l2.5}
	Let $\epsilon\in(0,1)$ and $p>1$. Define
\begin{align*}
	k:=\frac{p(p-1)\epsilon}{2(1+(p-1)\epsilon)},
	\quad
	l:=\frac{(p-1)(1-\epsilon)}{1+(p-1)\epsilon},\quad
	m:=
	\frac{\sqrt{p\bigl(1+(p-1)\epsilon-p\epsilon^2\bigr)}}
	{2(1-\epsilon)}.
\end{align*}
	Let
	\[
	\phi(x):=e^{kx}\cos(m x)^{-l},
	\qquad 0\le x\le N.
	\]
	If
	\begin{equation}
		N<\frac{\pi}{2m}=
		\frac{\pi(1-\epsilon)}
		{\sqrt{p\bigl(1+(p-1)\epsilon-p\epsilon^2\bigr)}},
		\label{l6.1}
	\end{equation}
	then $\phi$ is well defined on $[0,N]$ and satisfies
\begin{align*}
	1\le \phi(x)\le \phi(N),
	\qquad
	0\le \phi'(x)<\infty,
	\qquad 0\le x\le N,
\end{align*}
	as well as
	\begin{equation}
		\frac1p\phi''(x)-\phi'(x)\ge0,
		\qquad 0\le x\le N.
		\label{l6.2}
	\end{equation}
	Moreover,
	\begin{equation}
		|(p-1)\phi(x)-2\phi'(x)|
		-
		2\sqrt{
			(p-1)(1-\epsilon)\phi(x)
			\left(
			\frac1p\phi''(x)-\phi'(x)
			\right)
		}
		=0
		\label{l6.3}
	\end{equation}
	for all $x\in[0,N]$.
\end{lemma}
\begin{proof}
Since $k>0$, $l>0$, and $0\le m x<\frac{\pi}{2}$ for all $x\in[0,N]$, we have $\cos(m x)>0$, which ensures that $\phi(x)=e^{kx}\cos(m x)^{-l}$ is well defined, strictly positive, and finite on $[0,N]$. Moreover, since $\tan(m x)\ge 0$ on $[0,N]$, it follows that
\begin{align}
\frac{\phi'(x)}{\phi(x)} = k + lm \tan(m x) \ge k > 0, \label{l6.4}
\end{align}
so that $\phi'(x)>0$ on $[0,N]$. Consequently, $\phi$ is strictly increasing and satisfies
\begin{align*}
1\le \phi(x)\le \phi(N), \qquad x\in[0,N].
\end{align*}
Since $\tan(m x)$ is finite on $[0,N]$, which implies
\begin{align*}
0\le \phi'(x)<\infty, \quad \text{for all } x\in[0,N].
\end{align*}
Set $y(x):=k+lm\tan(mx)$, \eqref{l6.4} gives $\phi'(x)=\phi(x)y(x)$. Differentiating, $\phi''(x)=\phi(x)\bigl(y'(x)+y(x)^2\bigr)$, where $y'(x)=lm^2(1+\tan^2(mx))$
and hence
\begin{align}
\frac1p\phi''(x)-\phi'(x)
=
\phi(x)\left(\frac1p(y'(x)+y(x)^2)-y(x)\right).\label{l6.5}
\end{align}
Because $\phi(x)>0$, it suffices to show $y(x)^2-py(x)+y'(x)\geq 0$. Let $z=\tan(mx)$, therefore
\begin{align}
y(x)^2-py(x)+y'(x)&=\Lambda_1z^2+\Lambda_2z+\Lambda_3\nonumber\\
&=\frac{p(p-1)\left(\sqrt{p(1+p \epsilon)}\:z-\sqrt{(1-\epsilon)}\right)^2}{4 (1+(p-1) \epsilon )^2},\label{l6.6}
\end{align}
where
\begin{align*}
\Lambda_1&=lm^2(l+1)=\frac{p^2(p-1)(1+p \epsilon)}{4 (1+(p-1) \epsilon )^2}>0,\\
\Lambda_2&=lm(2k-p)=-\frac{p(p-1)\sqrt{p (1-\epsilon) (1+p \epsilon)}}{2 (1+(p-1) \epsilon)^2}<0,\\
\Lambda_3&=k^2-pk+lm^2=\frac{p(p-1)(1-\epsilon)}{4 (1+(p-1) \epsilon)^2}>0.
\end{align*}
Since $\Lambda_1>0$ and $\Lambda_2^2-4\Lambda_1\Lambda_3=0$, \eqref{l6.5} is non-negative, for all $0\leq x\leq N$. From \eqref{l6.6}, direct computations shows that
%Moreover, using the identity satisfied by $y$ (which is obtained from the explicit construction of $\phi$ via the associated Riccati equation), we have
\begin{align*}
\frac1p(y'+y^2)-y
=
\frac{\bigl((p-1)-2y\bigr)^2}{4(p-1)(1-\epsilon)}.
\end{align*}
Substitute in to \eqref{l6.5}, yields
\begin{align*}
\frac1p\phi''-\phi'
=
\frac{\bigl((p-1)\phi-2\phi'\bigr)^2}{4(p-1)(1-\epsilon)\phi},
\end{align*}
which is non-negative. Taking square root directly gives \eqref{l6.6}.
\end{proof}

This further lemma is, conversely, employed in the analysis of Theorem \ref{t1.3}.
\begin{lemma}\label{l2.6}
	Let $\Omega$ be a bounded convex domain in $\mathbb{R}^3$. Then for any nonnegative $\mathcal{F} \in \mathcal{C}^1(\Omega)$ and for every $\omega>0$ it holds
	\begin{align}
		\into \mathcal{F}^3 \leq A_1 \left(\into \mathcal{F}^2\right)^\frac{3}{2} 
		+ \frac{A_2}{\omega^3} \left(\into \mathcal{F}^2\right)^3
		+ A_3 \omega \into \lvert\nabla \mathcal{F}\rvert^2,\label{l2.6.1}
	\end{align}
	where 
	\begin{align*}
		A_1 = \frac{3^{\frac32}}{2 \rho^\frac32},
		\quad
		A_2 = \frac{3^3}{4^\frac{15}{4}}\left(1 + \frac{d}{\rho}\right)^\frac32,
		\quad 
		A_3 = \sqrt{2} \left(1 + \frac{d}{\rho}\right)^\frac32.
	\end{align*}
	and 
	\begin{align*}
		\rho := \min_{\partial\Omega}{x \cdot \nu } > 0, \qquad d := \max_{\partial\Omega}{\lvert x \rvert}.
	\end{align*}
	\begin{proof}
		The proof is a combination of the result in  \cite[Lemma A.2]{lepayne}   and known inequalities. (See details in \cite{VIGLIALORO-JMAA-BlowUp-Attr-Rep}). 
	\end{proof}
\end{lemma}

\section{A priori estimates and global boundedness} \label{sec3}
\subsection{The case $\tau=1$} 
This section is devoted to the global boundedness of solutions to system \eqref{1.1}-\eqref{1.3}. The following lemma yields an $\lps$ estimate for the second component.

\begin{lemma}\label{l3.1}
	Let $(u_0, v_0, w_0)$ be initial data satisfying \eqref{1.4}. For any $p>1$ and $\epsilon\in(0,1)$, assume that \eqref{l6.1} holds with $N=\chi K_1$, where $K_1$ is defined in \eqref{l2.3.1}. Then there exists a constant $C>0$ such that
	\begin{equation*}
		\nvt_{\lps}\leq C,
		\hspace*{1cm}
		\forall\, t\in(0,\tmax).
	\end{equation*}
\end{lemma}
\begin{proof}
	Set $\vo=\chi u$ to the function $\phi(x)$ defined in Lemma \ref{l2.5}. Multiply with $\upo$, $p>1$, the second equation of \eqref{1.1} and integrate with respect to $\Omega$ to get
	\begin{align*}
		\frac{1}{p}\dt\ints\bup\phvo=&\ints\upo v_t\phvo+\frac{1}{p}\ints\bup\povo \vo_t\\
		=&\ints\upo\phvo\left(\lv-\nabla\cdot(v \gvo)+\sigma_1v\left(1-v+\frac{\vo}{\chi}\right)-\delta_1 v\right)\\
		&+\frac{1}{p}\ints\bup\povo\Big(\Delta\vo-\vo(\alpha+\beta v+\gamma w)+\chi f(x,t)\Big).
	\end{align*}
	Using the integration by parts to the above estimate leads to
	\begin{align*}
		\frac{1}{p}\dt\ints\bup\phvo\leq &-(p-1)\ints v^{p-2}\phvo\mgv^2-\ints\upo\povo\gv\cdot\gvo\\
		&+(p-1)\ints\upo\phvo\gv\cdot\gvo+\ints\bup\povo\mgvo^2-\ints\upo\povo\gv\cdot\gvo\\
		&-\frac{1}{p}\ints\bup\ptvo\mgvo^2+\sigma_1\ints\bup \phvo+\frac{\sigma_1}{\chi} \ints \vo v^p\phvo+\frac{\chi M}{p}\ints \bup\povo.
	\end{align*}
	Combining the terms, one can get
	\begin{align}
			\frac{1}{p}\dt\ints\bup\phvo+&(p-1)\epsilon\ints v^{p-2}\phvo\mgv^2\leq -(p-1)(1-\epsilon)\ints v^{p-2}\phvo\mgv^2 \nonumber\\
			&+\ints\left((p-1)\phvo-2 \povo\right)\upo\gv\cdot\gvo +(\sigma_1+\sigma_1\nru_{\lis})\ints\bup\phvo\nonumber  \\
			&-\ints\left(\frac{1}{p}\ptvo-\povo\right)\bup\mgvo^2 +\frac{\chi M}{p}\ints \bup\povo,\label{l4.2} 
	\end{align} 
	for all $t\in(0, \tmax)$. Using the Gagliardo-Nirenberg inequality  and the  Young inequality, we have
	\begin{align*}
		c_1\ints\bup=c_1\left\|v^{\frac{p}{2}}\right\|^2_{\lts}&\leq c_2\left(\left\|\nabla v^{\frac{p}{2}}\right\|^{2r_0}_{\lts}\,\,\left\|v^{\frac{p}{2}}\right\|^{2(1-r_0)}_{\ltps}+\left\|v^{\frac{p}{2}}\right\|^2_{\ltps}\right) \\
		&\leq \kappa\left(\left\|\nabla v^{\frac{p}{2}}\right\|^{2r_0}_{\lts}\right)^\frac{1}{r_0}+c_2\,c(\kappa)\left(\left\|v^{\frac{p}{2}}\right\|^{2(1-r_0)}_{\ltps}\right)^\frac{1}{1-r_0}+c_2\left\| v^{\frac{p}{2}}\right\|^2_{\ltps} \\
		&\leq  \kappa\left\|\nabla v^{\frac{p}{2}}\right\|^{2}_{\lts}+c_3 \\
		&\leq \frac{\kappa p^2}{4}\ints v^{p-2}\mgv^2+c_{3}.
	\end{align*}
	Thanks to $\phvo\geq 1$,
	\begin{equation} \label{l4.3} \begin{split}
			c_1\ints\bup\leq \frac{\kappa p^2}{4}\ints v^{p-2}\phvo\mgv^2+c_{3}.
	\end{split} \end{equation} 
	with $c_3=c_2c(\kappa) K^p_3, \: c(\kappa)=\left(\frac{\kappa}{r_0}\right)^{\frac{-r_0}{1-r_0}}(1-r_0)$,
	where $r_0=\frac{\frac{p}{2}-\frac{1}{2}}{\frac{p}{2}+\frac{1}{n}-\frac{1}{2}}\in(0,1)$.  
	From \eqref{l4.3}, we can estimate
	\begin{equation}\label{l4.4} \begin{split}
			\left(\sigma_1+\sigma_1\nru_{\lis}\|\phi\|_{\lis}+\frac{\chi M}{p}\|\phi'\|_{\lis}\right)\ints\bup
			\leq \frac{(p-1)\epsilon}{2}\ints v^{p-2}\phvo\mgv^2+c_{4}.
	\end{split} \end{equation} 
	%	Using the Young's inequality, we have
	%	\begin{equation} \begin{split}
			%	\mu_1\ints\upk\phvo %\leq&\:\mu_1 \phi(\chi \|v\|_{\lis})\ints\upk\\
			%	\leq & \mu_1\ints\bu^{p+\kappa}\phvo+c_{5}(\phi(K)),\label{l4.5}
			%	\end{split} \end{equation} 
	%	where $c_5>0$. 
	Substitute \eqref{l4.4} in \eqref{l4.2}, to obtain
	\begin{equation} \label{l4.6} 
		\begin{split}
			\frac{1}{p}\dt\ints\bup\phvo+\frac{(p-1)\epsilon}{2}\ints v^{p-2}\phvo\mgv^2\leq&\:-(p-1)(1-\epsilon)\ints v^{p-2}\phvo\mgv^2 \\
			&+\ints\left((p-1)\phvo-2\povo\right)\upo\gv\cdot\gvo \\
			&-\ints\left(\frac{1}{p}\ptvo-\povo\right)\bup\mgvo^2+c_4.
	\end{split} \end{equation} 
	Again, using \eqref{l4.3}, we have
	\begin{equation} \begin{split}
			\ints\bup\phvo\leq\phi(\chi\nru_{\lis})\ints\bup\leq \frac{(p-1)\epsilon}{2}\ints v^{p-2}\phvo\mgv^2+c_{5}.\label{l4.7}
	\end{split} \end{equation} 
	Substitute \eqref{l4.7} in \eqref{l4.6} to obtain
	{\small
		\begin{align*}
			\frac{1}{p}\dt\ints\bup\phvo+\ints\bup\phvo\leq&-(p-1)(1-\epsilon)\ints v^{p-2}\phvo\mgv^2-\ints\left(\frac{1}{p}\ptvo-\povo\right)\bup\mgvo^2\\
			&+\ints\left((p-1)\phvo-2\povo\right)\upo\mgv\mgvo+c_6\\
			%\leq&\: -\ints\left(\sqrt{(p-1)(1-\epsilon)\phvo} v^{\frac{p-2}{2}}\mgv\right)^2 -\ints\left(\sqrt{\frac{1}{p}\ptvo-\povo} v^\frac{p}{2}\mgvo\right)^2\\
			%&+\ints\left((p-1)\phvo-2\povo\right)\upo\mgv\mgvo+c_6\\
			\leq&\: -\ints\left[\left(\sqrt{(p-1)(1-\epsilon)\phvo} v^{\frac{p-2}{2}}\mgv\right)^2+\left(\sqrt{\frac{1}{p}\ptvo-\povo} v^\frac{p}{2}\mgvo\right)^2\right.\\
			&-\left.2\sqrt{(p-1)(1-\epsilon)\phvo\left(\frac{1}{p}\ptvo-\povo\right)} v^{\frac{p-2}{2}+\frac{p}{2}}\mgv\mgvo\right]\\
			&-\ints2\sqrt{(p-1)(1-\epsilon)\phvo\left(\frac{1}{p}\ptvo-\povo\right)} v^{\frac{p-2}{2}+\frac{p}{2}}\mgv\mgvo\\
			&+\ints\left((p-1)\phvo-2\povo\right)\upo\mgv\mgvo+c_6.
	\end{align*}}
	Combining the terms in the above inequality, one gets
	\begin{align*}
		\frac{1}{p}\dt\ints\bup\phvo+&\ints\bup\phvo\leq-\ints\left[\sqrt{(p-1)(1-\epsilon)\phvo} v^{\frac{p-2}{2}}\mgv-\sqrt{\frac{1}{p}\ptvo-\povo} v^\frac{p}{2}\mgvo\right]^2\\
		&+\ints\left((p-1)\phvo-2\povo-2\sqrt{(p-1)(1-\epsilon)\phvo\left(\frac{1}{p}\ptvo-\povo\right)}\right)\\
		&\times v^{p-1}\mgv\mgvo+c_6,
		%\leq&\: -\ints\left[\sqrt{(p-1)(1-\epsilon)\phvo} v^{\frac{p-2}{2}}\mgu-\sqrt{\frac{1}{p}\ptvo-\povo} v^\frac{p}{2}\mgvo\right]^2\\
		%&+\ints\psi v^{p-1}\mgv\mgvo+c_6,
	\end{align*}
	where $(p-1)\phvo-2\povo-2 \sqrt{ (p-1)(1-\epsilon)\phvo\left(\frac{1}{p}\ptvo-\povo\right)}=0$, for all $0\leq \vo\leq\chi K_1$. By ODE arguments, we can conclude that
	\begin{equation*} 
		\frac{1}{p}\ints\bup\phvo\leq c_6, \qquad \text{for all}\:\: t\in(0, \tmax).
	\end{equation*} 
	where $c_6>0$. This completes the proof.
\end{proof}

\subsection{The case $\tau=0$} 
In this section, we consider system \eqref{1.1}--\eqref{1.3} in the elliptic case, i.e., $\tau=0$, where the first and third equations are stationary. For clarity, we focus on the following simplified model
\begin{align}
\left\{
\begin{array}{llll}
	0= \Delta u-u(\alpha+\beta v+\gamma w)+f(\cdot,t),\hspace*{0.5cm} &\text{in} \; \Omega\times(0,\tmax),\\
	v_t= \Delta v-\chi\nabla\cdot(v\gu)+\sigma_1 v(1-v+u)-\delta_1 v, &\text{in} \; \Omega\times(0,\tmax),\\
	0= \Delta w+\sigma_2 uw-\delta_2 w, &\text{in} \; \Omega\times(0,\tmax),\\
	\frac{\partial u}{\partial\nu}=\frac{\partial v}{\partial\nu}=\frac{\partial w}{\partial\nu}=0, &\text{on} \; \partial\Omega\times(0,\tmax),\\
	v(x,0)=v_0, &x\in\overline{\Omega}.
\end{array}
\right.\label{l3.2.1}
\end{align}
As observed above, suitable a priori estimates for $v$ are sufficient to guarantee global boundedness. In particular, we establish the following result.

\begin{lemma}\label{l3.2}
	For every $p>1$, there exists a constant $C>0$ such that the local solution $(u,v,w)$ of \eqref{l3.2.1} satisfies
	\begin{align*}
		\nvt_{\lps}\leq C,
		\qquad \text{for all } t\in(0,\tmax).
	\end{align*}
\end{lemma}
	\begin{proof}
		We begin with the differentiation of the functional $\Phi(t)=\into v^p$ where we apply the divergence theorem 
		\begin{equation*}
			\begin{split}
				\Phi'(t)=&p\into v^{p-1}\Delta v - p\chi\into v^{p-1}\nabla\cdot \left( v\nabla u\right)+p\sigma_1\into v^p(1-v+u)-p\delta_1 \into v^p\\
				\leq&-p(p-1)\into v^{p-2}\left|\nabla v \right|^2+p(p-1)\chi\into v^{p-1}\nabla u\cdot\nabla v+p\sigma_1(1+\nru_{\lis})\ints v^p\\
				\leq&-\frac{4(p-1)}{p}\into \left|\nabla v^{\frac p2}\right|^2-(p-1)\chi\into v^p\Delta u +p\sigma_1(1+\nru_{\lis})\ints v^p
			\end{split}
		\end{equation*}
		We now make use of the second equation in conjunction with the fact that $u,v$ and $w$ are positive and $w$ is bounded, so to write
		\begin{align}\label{l3.2.2}
				\Phi'(t)=&-\frac{4(p-1)}{p}\into \left|\nabla v^{\frac p2}\right|^2-(p-1)\chi\into v^pu\left(\alpha+\beta v+\gamma w\right)+ (p-1)\chi\into v^p f(\cdot,t)\nonumber\\
				&+p\sigma_1(1+\nru_{\lis})\ints v^p\nonumber\\
				\leq&-\frac{4(p-1)}{p}\into \left|\nabla v^{\frac p2}\right|^2+\left((p-1)\chi M+p\sigma_1(1+\nru_{\lis})\right)\into v^p
		\end{align}
we can say that, for proper positive constant $c_1$ depending on $p$,
		\begin{equation}\label{gn}
			\begin{split}
	\left((p-1)\chi M+p\sigma_1(1+\nru_{\lis})\right)\into v^p
				\leq&\frac{2(p-1)}{p} \into |\nabla v^{\frac p2} |^2 +c_1 \quad \text{on $(0,\tmax)$,}
			\end{split}
		\end{equation}
		where in the last step we applied \eqref{l4.3}. In this way \eqref{l3.2.2} is turned into 
		\begin{equation*}
			\Phi'(t)\leq-\frac{2(p-1)}{p}\into \lvert\nabla v^{\frac p2}\rvert^2+c_1 \quad \text{on $(0,\tmax)$,}
		\end{equation*}
		which from \eqref{gn} gives this initial problem with $c_2, c_3>0$
		%    Again relying on the Gagliardo--Nirenberg inequality, making analogous calculations, we arrive at this initial problem, for some positive constants $c_4$ and $c_5$
		\begin{align*}
			\left\{
			\begin{array}{llll}
				\Phi'(t)\leq c_3-c_2\Phi(t) & \text{on $(0,\tmax)$},\\
				\Phi(0)=\into v_0^p.
			\end{array}
			\right.
		\end{align*}
		Naturally, the above problem is solvable and $\Phi(t)\leq\displaystyle{ \max \left\{\into v_0^p, \frac{c_3}{c_2}\right\}}$.
\end{proof}

\begin{proof}[Proof of Theorem \ref{t1.1}]
	We first consider the case $\tau=0$. In this case, the conclusion follows directly from Lemmas \ref{l3.2} and \ref{l2.4}. Next, let $\tau=1$. To apply Lemma \ref{l3.1} (and subsequently Lemma \ref{l2.4}), it suffices to verify condition \eqref{l6.1}. Let $	N:=\chi K_1< \pi\sqrt{\frac{2}{n}}$. Choose
	\begin{align*}
		\epsilon
		=
		\frac{\pi^2-\frac{n}{2}N^2}
		{\pi^2+\left(\frac{n}{2}\right)^2N^2}\: \in (0, 1).
	\end{align*}
 Moreover, $p>\frac{n}{2}$ yields
	\begin{align*}
		N
		<
		\frac{\pi(1-\epsilon)}
		{\sqrt{p\bigl(1+(p-1)\epsilon-p\epsilon^2\bigr)}}
		=
		\frac{\pi}{2m}.
	\end{align*}
	Hence condition \eqref{l6.1} holds. Lemma \ref{l3.1} now applies, and combining it with Lemma \ref{l2.4} gives the desired conclusion.
\end{proof}

\section{Lower bound for the finite-time blow-up of solutions in \texorpdfstring{$\mathbb{R}^{3}$}{R3}}
\label{sec4}
This section is devoted to deriving a lower bound for the maximal existence time of solutions in $\mathbb{R}^3$. In this direction, we rely on the following result.

\begin{lemma}\label{l4.1}
	Suppose that $\Omega\subset\mathbb{R}^3$ is a bounded convex domain and that $\tau\in\{0,1\}$. Let $(u,v,w)$ be a solution of \eqref{1.1} that blows up at the finite time $\tmax$, in the sense that
	\begin{align*}
		\limsup_{t\to \tmax}\|v(\cdot,t)\|_{L^\infty(\Omega)}=+\infty.
	\end{align*}
	Then there exist computable constants $\mathcal{A}_\tau$, $\mathcal{B}_\tau$, $\mathcal{C}_\tau$, and $\mathcal{D}_\tau$ such that
	\begin{equation}
		\Psi_\tau'(t)
		\leq
		\mathcal{A}_\tau \Psi_\tau(t)^3
		+ \mathcal{B}_\tau \Psi_\tau(t)^{\frac32}
		+ \mathcal{C}_\tau \Psi_\tau(t)^\tau
		+ \tau \mathcal{D}_\tau,
		\qquad t\in(0,\tmax),
		\label{l4.1.1}
	\end{equation}
	where $\Psi_\tau$ is defined in \eqref{t1.3.1}.
\end{lemma}

\begin{proof}
%	We will distinguish the fully parabolic case, i.e., $\tau=1$, and the elliptic-parabolic-elliptic case, i.e., $\tau=0$.\\
	The case $\tau=1$. We have to deal with the following  energy function 
	\begin{align}
		\Psi_1(t) := \ints \mgu^4 + \ints v^2  + \ints \mgw^2, \qquad \text{for all}\; t\in(0, \tmax). \label{l4.1.2}
	\end{align}
Let us denote the integrals in \eqref{l4.1.2} by $I_1(t)$, $I_2(t)$, and $I_3(t)$ respectively. Let us then compute their derivatives one at a time. For the first integral $I_1(t)$
\begin{align*}
\dt I_1=&4\ints \mgu^2\: \gu\cdot\nabla\left(\lu-\alpha u-\beta uv-\gamma uw+f\right) \\
	= & 4\ints \mgu^2\:\gu\cdot \nabla\lu-4\alpha\ints \mgu^4-4\beta\ints u \mgu^3\:\gv-4\beta\ints v \mgu^4\\
	&-4\gamma\ints u \mgu^3\:\gw-4\gamma\ints w \mgu^4+4\ints \mgu^3\:\nabla f.
\end{align*}
Moreover by virtue of the identity $2\gu\cdot\nabla\lu=\Delta(\mgu^2)-2|D^2u|^2$ and neglecting the  non-negative terms, we get
\begin{align*}
	\dt I_1\leq & 2\ints \mgu^2\Delta(\mgu^2)-4\ints \mgu^2|D^2u|^2+4\beta\nru_{\lis}\ints  \mgu^3\:\gv\\
	&+4\gamma\nru_{\lis}\ints \mgu^3\:\gw+4\ints \mgu^3\:\nabla f.
	\end{align*}
Using Young's inequality to the third and fourth term of the above inequality, we get
\begin{align*}
	\dt I_1\leq & 2\ints \mgu^2\Delta(\mgu^2)-4\ints \mgu^2|D^2u|^2+\ints \mgv^2+4\beta^2\nru_{\lis}^2\ints  \mgu^6\\
	&+\delta_2\ints\mgw^2+\frac{4\gamma^2\nru_{\lis}^2}{\delta_2}\ints \mgu^6+\ints\mgu^6+4M^2\momeg.
\end{align*}
Combining the terms, one can get
\begin{align*}
	\dt I_1\leq & 2\ints \mgu^2\Delta(\mgu^2)-4\ints \mgu^2|D^2u|^2+\ints \mgv^2+\delta_2\ints\mgw^2+B_1\ints \mgu^6+4M^2\momeg.
\end{align*}
where $B_1:=4\beta^2\nru_{\lis}^2+\frac{4\gamma^2\nru_{\lis}^2}{\delta_2}+1$. 
Recall \ref{l2.6} with the choice of $\zeta_1=\frac{1}{2A_3}$, we obtain
\begin{align}
	\dt I_1\leq & -2\ints |\nabla\mgu^2|^2-4\ints \mgu^2|D^2u|^2+\ints \mgv^2+\delta_2\ints\mgw^2+B_1A_1\left(\ints\mgu^4\right)^\frac32\nonumber\\
	&+\frac{B_1A_2}{\zeta_1^3}\left(\ints\mgu^4\right)^3+B_1A_3\zeta_1\ints|\nabla\mgu^2|^2+4M^2\momeg\nonumber\\
	\leq & -4\ints \mgu^2|D^2u|^2+\ints \mgv^2+\delta_2\ints\mgw^2+B_1A_1\left(\ints\mgu^4\right)^\frac32\nonumber\\
	&+\frac{B_1^4A_2A_3^3}{8}\left(\ints\mgu^4\right)^3+4M^2\momeg.\label{l4.1.3}
\end{align}
With the second term $I_2(t)$ we obtain, once again by relying on the divergence theorem
\begin{align}
	\dt I_2\leq& -2\ints v\lv-2\chi\ints v\nabla\cdot(v\gu)+2\sigma_1(1+\nru_{\lis})\ints v^2-2\sigma_1\ints v^3\nonumber\\
	\leq & -2\ints\mgv^2+2\chi\ints v\gu\cdot\gv+2\sigma_1(1+\nru_{\lis})\ints v^2.\label{l4.1.4}
\end{align}
Using Young's inequality, we get
\begin{align*}
	\dt I_2\leq&  -2\ints\mgv^2+\frac12\ints \mgv^2+2\chi^2\ints v^3+\frac{8\chi^2}{27}\ints \mgu^6+2\sigma_1\ints v^3\\
	&+\frac{8\sigma_1}{27}(1+\nru_{\lis})^3\momeg-2\sigma_1\ints v^3\\
	\leq & -\frac32\ints \mgv^2+2\chi^2\ints v^3+\frac{8\chi^2}{27}\ints \mgu^6+\frac{8\sigma_1}{27}(1+\nru_{\lis})^3\momeg.
\end{align*}
Recall \ref{l2.6} with the choice $\zeta_2=\frac{1}{4\chi^2A_3}$ and $\zeta_3=\frac{1}{16\chi^2A_3}$, we obtain
\begin{align}
	\dt I_2	\leq & -\frac32\ints \mgv^2+2\chi^2A_1\left(\ints v^2\right)^\frac32+\frac{2\chi^2A_2}{\zeta_2^3}\left(\ints v^2\right)^3+2\chi^2A_3\zeta_2\ints \mgv^2\nonumber\\
	&+\frac{8\chi^2A_1}{27}\ints \left(\mgu^4\right)^\frac32+\frac{8\chi^2A_2}{27\zeta_3^3}\ints \left(\mgu^4\right)^3+8\chi^2A_3\zeta_3\ints|\nabla\mgu^2|^2\nonumber\\
	&+\frac{8\sigma_1}{27}(1+\nru_{\lis})^3\momeg\nonumber\\
	\leq & -\ints \mgv^2+2\chi^2A_1\left(\ints v^2\right)^\frac32+2^7\chi^8A_2A_3^3\left(\ints v^2\right)^3\nonumber\\
	&+\frac{8\chi^2A_1}{27}\ints \left(\mgu^4\right)^\frac32+\frac{2^{16}\chi^8A_2A_3^3}{27}\ints \left(\mgu^4\right)^3+\frac12\ints|\nabla\mgu^2|^2\nonumber\\
	&+\frac{8\sigma_1}{27}(1+\nru_{\lis})^3\momeg.\label{l4.1.5}
\end{align}
With analogous steps, last integral $I_3(t)$ leads to 
\begin{align*}
	\dt I_3=&2\ints \gw\cdot\nabla\Delta w+2\sigma_2 \ints \gw\cdot u\gw+2\sigma_2\ints \gw\cdot w\gu-2\delta_2\ints \mgw^2
\end{align*}
Using Young inequality
\begin{align*}
	\dt I_3\leq& \ints \Delta(\mgw^2)-2\ints |D^2w|^2+2\sigma_2 \nru_{\lis}\ints \mgw^2+2\sigma_2\nrw_{\lis}\ints \gw\cdot \gu\\
	&-2\delta_2\ints \mgw^2\\
	\leq& \ints \Delta(\mgw^2)-2\ints |D^2w|^2+2\sigma_2 \nru_{\lis}\ints \mgw^2+\delta_2\ints \mgw^2\\
	&+\frac{\sigma_2\nrw_{\lis}}{\delta_2}\ints \mgu^2-2\delta_2\ints \mgw^2\\
	\leq& \ints \Delta(\mgw^2)-2\ints |D^2w|^2+2\sigma_2 \nru_{\lis}\ints \mgw^2+\frac{\sigma_2\nrw_{\lis}}{\delta_2}\ints \mgu^2-\delta_2\ints \mgw^2.
\end{align*}
Again using Young inequality,
\begin{align*}
	\dt I_3\leq& 2\sigma_2 \nru_{\lis}\ints \mgw^2+\ints \mgu^6-\delta_2\ints \mgw^2+\frac{2\sigma_2^3\nrw_{\lis}^3}{3^\frac32\delta_2^\frac32}\momeg.
\end{align*}
Recall \ref{l2.6} with the choice $\zeta_4=\frac{1}{2A_3}$, we get
\begin{align}
	\dt I_3\leq& 2\sigma_2 \nru_{\lis}\ints \mgw^2
	+A_1\left(\ints \mgu^4\right)^\frac32+\frac{A_2}{\zeta_4^3}\left(\ints\mgu^4\right)^3+A_3\zeta_4\ints|\nabla\mgu^2|^2\nonumber\\
	&-\delta_2\ints \mgw^2+\frac{2\sigma_2^3\nrw_{\lis}^3}{3^\frac32\delta_2^\frac32}\momeg\nonumber\\
	\leq& 2\sigma_2 \nru_{\lis}\ints \mgw^2
	+A_1\left(\ints \mgu^4\right)^\frac32+2^3A_2A_3^3\left(\ints\mgu^4\right)^3+\frac12\ints|\nabla\mgu^2|^2\nonumber\\
	&-\delta_2\ints \mgw^2+\frac{2\sigma_2^3\nrw_{\lis}^3}{3^\frac32\delta_2^\frac32}\momeg.\label{l4.1.6}
\end{align}
Combining \eqref{l4.1.3}, \eqref{l4.1.5} and \eqref{l4.1.6}, we obtain 
\begin{align*}
	\dt \Psi_1	\leq & -4\ints \mgu^2|D^2u|^2+\ints|\nabla\mgu^2|^2+\left(B_1A_1+\frac{8\chi^2A_1}{27}+A_1\right)\left(\ints \mgu^4\right)^\frac32\\
	&+\left(\frac{B_1^4A_2A_3^3}{8}+\frac{2^{16}\chi^8A_2A_3^3}{27}+2^3A_2A_3^3\right)\left(\ints\mgu^4\right)^3+2\chi^2A_1\left(\ints v^2\right)^\frac32\\
	&+2^7\chi^8A_2A_3^3\left(\ints v^2\right)^3+2\sigma_2 \nru_{\lis}\ints \mgw^2+\frac{8\sigma_1}{27}(1+\nru_{\lis})^3\momeg\\
	&+\frac{2\sigma_2^3\nrw_{\lis}^3}{3^\frac32\delta_2^\frac32}\momeg+4M^2\momeg.
\end{align*}
We invoked as well the inequality $|\nabla\mgu^2|^2\leq 4\mgu^2|D^2u|^2$ applied to the second term in RHS, 
\begin{align*}
	\dt \Psi_1	\leq & \left(B_1A_1+\frac{8\chi^2A_1}{27}+A_1\right)\left(\ints \mgu^4\right)^\frac32+2\chi^2A_1\left(\ints v^2\right)^\frac32\\
	&+\left(\frac{B_1^4A_2A_3^3}{8}+\frac{2^{16}\chi^8A_2A_3^3}{27}+2^3A_2A_3^3\right)\left(\ints\mgu^4\right)^3+2^7\chi^8A_2A_3^3\left(\ints v^2\right)^3\\
	&+2\sigma_2 \nru_{\lis}\ints \mgw^2+\frac{8\sigma_1}{27}(1+\nru_{\lis})^3\momeg+\frac{2\sigma_2^3\nrw_{\lis}^3}{3^\frac32\delta_2^\frac32}\momeg+4M^2\momeg\\
	\leq &  \mathcal{A}_1 \Psi_1(t)^3 + \mathcal{B}_1 \Psi_1(t)^\frac32 + \mathcal{C}_1 \Psi_1(t)+\mathcal{D}_1,
\end{align*}
with
\begin{align*}
	\mathcal{A}_1:=& \max\left\{1, \frac{B_1^4A_2A_3^3}{8}+\frac{2^{16}\chi^8A_2A_3^3}{27}+2^3A_2A_3^3, 2^7\chi^8A_2A_3^3\right\},\\
	\mathcal{B}_1:=& \max \left\{1, B_1A_1+\frac{8\chi^2A_1}{27}+A_1, 2\chi^2A_1 \right\},\\
	\mathcal{C}_1:= & 2\sigma_2 \nru_{\lis},\\
	\mathcal{D}_1:= & \frac{8\sigma_1}{27}(1+\nru_{\lis})^3\momeg+\frac{2\sigma_2^3\nrw_{\lis}^3}{3^\frac32\delta_2^\frac32}\momeg+4M^2\momeg.
\end{align*}
For $\tau=0$, the simplification shows that among the terms $I_1$, $I_2$ and $I_3$ in \eqref{l4.1.2}, only $I_2$ contributes to the analysis, so that the energy functional reduces to
\begin{align*} 
	\Psi_0(t) := \into v^2\quad \text{for all} \; t \in (0,\tmax).
\end{align*}
Subsequently,  from  \eqref{l4.1.4} we can write, by applying twice the divergence theorem,   
\begin{align}
	\dt I_2	\leq & -2\ints\mgv^2-\chi\ints v^2\lu+2\sigma_1(1+\nru_{\lis})\ints v^2
\end{align}
which thanks to the first equation in system \eqref{1.1} turns into
\begin{align*}
	\dt I_2	\leq & -2\ints\mgv^2-\chi\alpha\ints uv^2-\chi\beta\ints uv^3-\chi\gamma\ints uv^2w+\chi\ints f v^2\\
	&+2\sigma_1(1+\nru_{\lis})\ints v^2\\
	\leq & -2\ints\mgv^2+\left(\chi M+2\sigma_1(1+\nru_{\lis})\right)\ints v^2\\
		\leq & -2\ints\mgv^2+\ints v^3+\frac{4\momeg}{27}\left(\chi M+2\sigma_1(1+\nru_{\lis})\right)^3.
\end{align*}
Here, Young's inequality has been used in the final step. As the case $\tau=1$ has already been treated extensively, we omit the corresponding details. Using \eqref{l2.6.1}, one can estimate $\ints v^3$ in terms of $\ints \mgv^2$, $\left(\ints v^2\right)^3$ and $\left(\ints v^2\right)^\frac{3}{2}$. It follows that there exist positive constants $\mathcal{A}_0$, $\mathcal{B}_0$ and $\mathcal{C}_0$ such that
\begin{align*}
	\Psi_0'(t) \leq \mathcal{A}_0 \Psi_0(t)^3  + \mathcal{B}_0 \Psi_0(t)^\frac32 + \mathcal{C}_0 \quad \textrm{on} \quad (0,\tmax).
\end{align*} 	
This completes the proof.
\end{proof}

\begin{proof}[Proof of Theorem \ref{t1.3}]
	Let $\tau\in\{0,1\}$. Integrating \eqref{l4.1.1} over $(0,t)$, we obtain
	\begin{align}
		t
		\geq
		\int_{\Psi_\tau(0)}^{\Psi_\tau(t)}
		\frac{\mbox{d}\Psi}
		{\mathcal{A}_\tau \Psi^3
			+\mathcal{B}_\tau \Psi^{\frac32}
			+\mathcal{C}_\tau \Psi^\tau
			+\tau\mathcal{D}_\tau }.
		\label{t1.3.p1}
	\end{align}
	If $\Psi_\tau(t)$ blows up at time $\tmax$, then letting $t\to\tmax$ in \eqref{t1.3.p1} yields
	\begin{align}
		\tmax
		\geq
		\int_{\Psi_\tau(0)}^{+\infty}
		\frac{\mbox{d}\Psi}
		{\mathcal{A}_\tau \Psi^3
			+\mathcal{B}_\tau \Psi^{\frac32}
			+\mathcal{C}_\tau \Psi^\tau
			+\tau\mathcal{D}_\tau }.
		\label{t1.3.p2}
	\end{align}
	In view of Remark \ref{r2.1}, we have $\lim_{t\to\tmax}\Psi_\tau(t)=+\infty$,	which immediately yields \eqref{t1.3.3}. To derive the explicit lower bound in \eqref{t1.3.4}, we follow the argument of \cite[Remark 6.2]{marras2020}. Since $\lim_{t\to\tmax}\Psi_\tau(t)=+\infty$,
	there exists $t_0\in[0,\tmax)$ such that $\Psi_\tau(t)\geq \Psi_\tau(0)$ for all $t\in(t_0,\tmax)$. Consequently,
	\begin{align*}
	\Psi_\tau(t)^{\frac32}\leq\Psi_\tau(0)^{-\frac32}\Psi_\tau(t)^3,\: 
	\Psi_\tau(t)^\tau\leq\Psi_\tau(0)^{\tau-3}\Psi_\tau(t)^3,\:
	1\leq\Psi_\tau(0)^{-3}\Psi_\tau(t)^3,\quad t\in(t_0,\tmax).
	\end{align*}
	Therefore,
%	\[
%	\mathcal{A}_\tau \Psi_\tau(t)^3
%	+\mathcal{B}_\tau \Psi_\tau(t)^{\frac32}
%	+\mathcal{C}_\tau \Psi_\tau(t)^\tau
%	+\tau\mathcal{D}_\tau
%	\leq
%	\mathcal{E}_\tau\,\Psi_\tau(t)^3,
%	\]
	%Substituting this estimate into \eqref{t1.3.p2}, we obtain
	\[
	\tmax
	\geq
	\int_{\Psi_\tau(0)}^{+\infty}
	\frac{\mbox{d}\Psi}{\mathcal{E}_\tau\,\Psi^3}
	=
	\frac{1}{2\mathcal{E}_\tau\,\Psi_\tau(0)^2},
	\]
	where $\mathcal{E}_\tau=\mathcal{A}_\tau +\mathcal{B}_\tau\Psi_\tau(0)^{-\frac32}
	+\mathcal{C}_\tau\Psi_\tau(0)^{\tau-3}
	+\tau\mathcal{D}_\tau\Psi_\tau(0)^{-3}$. This proves \eqref{t1.3.4} and completes the proof.
\end{proof}

\begin{remark}
	We note that an analysis analogous to that of Theorem \ref{t1.3} can also be carried out for domains in $\mathbb{R}^n$ with $n\geq 4$. In this setting, however, inequality \eqref{l2.6.1} is no longer applicable, and an alternative estimate is required. To overcome this difficulty, one may employ suitable Sobolev embedding inequalities, following the approach developed in \cite{and2017}.
\end{remark}

\section{Numerical Experiments}\label{sec5}
In this section, we perform finite element computations for the Fish–Mussels chemotaxis system in the three-dimensional domain $\Omega=[-0.5,0.5]^3$. The spatial domain is uniformly discretized using tetrahedral $P_1$ linear finite elements for all unknown variables, while the temporal discretization is carried out using the backward Euler (implicit Euler) scheme. The numerical implementation is performed using Python 3.13.11 (FEniCSx 0.10.0), and MATLAB R2024a is used for visualization and plotting.

The fully coupled nonlinear system arising from the backward Euler finite element discretization is solved using the PETSc SNES Newton line-search solver with relative tolerance $10^{-6}$, absolute tolerance $10^{-8}$, and a maximum of 20 Newton iterations. The resulting linearized systems are solved using the GMRES iterative solver preconditioned with LU factorization.

The FEM scheme is verified through mesh convergence studies demonstrating the expected order of accuracy. In addition, the computed PDE solutions are validated against the corresponding spatially homogeneous ODE model solved using MATLAB \texttt{ode45}, showing excellent agreement in the long-time dynamics and validation.

\subsection{Convergence study}

To perform the convergence study of the proposed finite element scheme, we consider the dimensionless parameter values $\chi = 2,
\alpha = \beta = \gamma = \sigma_1 = \sigma_2 = \delta_1 = \delta_2 = 1, f = 1$, with the time step chosen as $\Delta t = 0.25\,\Delta x^2$.
The initial conditions are chosen as
\begin{align*}
	u(\cdot,0)=0,
	\qquad
	v(\cdot,0)=15,
	\qquad
	w(\cdot,0)=5.
\end{align*}
Furthermore, to verify the convergence behaviour of the numerical scheme and compare the errors across different mesh levels, we compute the discrete numerical errors in the $\lts$ and $\lis$ norms over the time interval $t \in [0,1]$.

The obtained numerical errors are illustrated in Figure~\ref{fig.5.1}. In Figure~\ref{fig.5.1}, the horizontal axis represents the logarithmic values of the degrees of freedom, while the vertical axis represents the logarithmic values of the corresponding errors. The linear behaviour observed in the log--log plots clearly demonstrates the second-order convergence of the proposed finite element scheme in both the $\lts$ and $\lis$ norms. Moreover, Tables~\ref{tab.5.1} present the computed numerical errors and the corresponding orders of convergence for each unknown variable at different levels of mesh refinement.
\begin{figure}[H]
	\centering
	\subfloat[]{\includegraphics[width=0.45\textwidth]{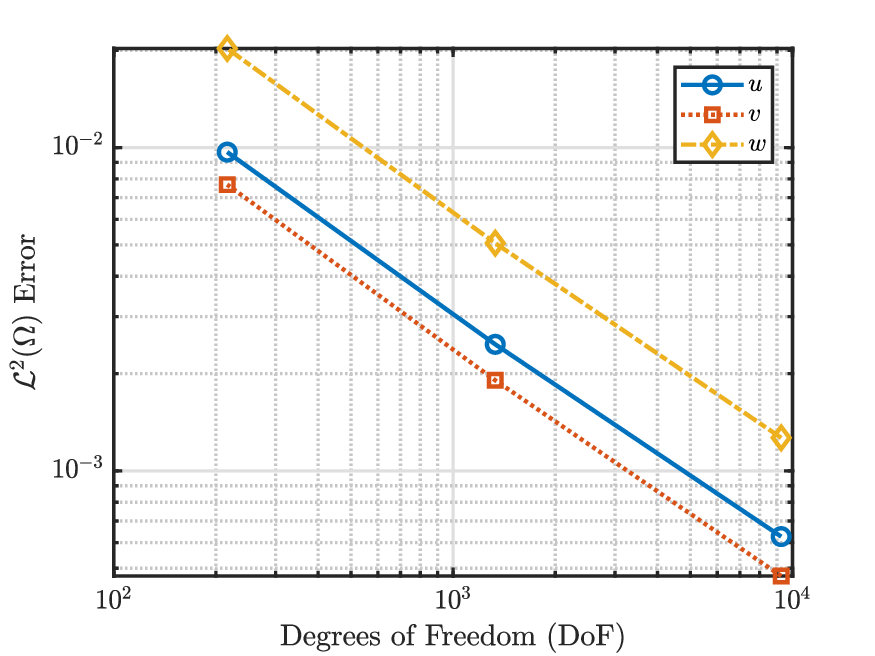}}
	\subfloat[]{\includegraphics[width=0.45\textwidth]{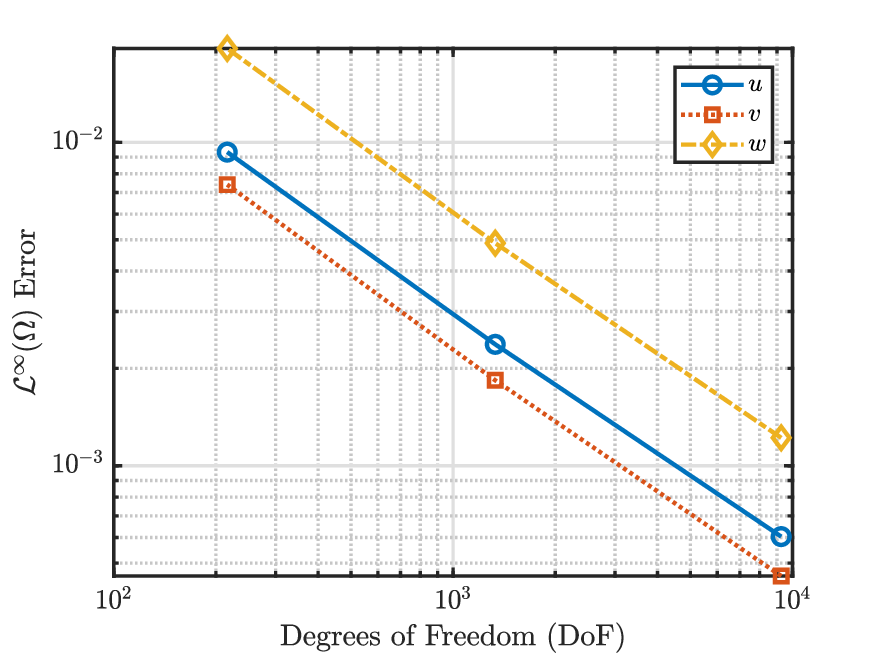}}
	\caption{Error plots of the unknown variables obtained at different mesh levels. Panels (A) and (B) represent the $\lts$ and $\lis$ errors versus the degrees of freedom (DOF), respectively.}
	\label{fig.5.1}
\end{figure}
\begin{table}[H]
	\centering
	\caption{Numerical error values and corresponding orders of convergence for $u(\cdot,t)$, $v(\cdot,t)$, and $w(\cdot,t)$ at different mesh levels.}
	\begin{tabular}{c c c c c c}
		\toprule
		{\bf Variable} & {\bf DOF} & {\bf $\lts$ error} & {\bf Order} & {\bf $\lis$ error} & {\bf Order} \\
		\midrule
		\multirow{2}{*}{$u$}
		& 216 to 1331  & 9.6822$\times 10^{-3}$ & ---    & 9.3302$\times 10^{-3}$ & ---    \\
		& 1331 to 9261 & 2.4630$\times 10^{-3}$  & 1.97491612  & 2.3735$\times 10^{-3}$ & 1.97491486\\
		\midrule
		
		\multirow{2}{*}{$v$}
		& 216 to 1331  & 7.6796$\times 10^{-3}$ & ---    & 7.4007$\times 10^{-3}$ & ---    \\
		& 1331 to 9261 & 1.9063$\times 10^{-3}$ & 2.01023606 & 1.8371$\times 10^{-3}$ & 2.01024513 \\
		\midrule
		
		\multirow{2}{*}{$w$}
		& 216 to 1331  & 2.0250$\times 10^{-2}$ & ---    & 1.9514$\times 10^{-2}$ & ---    \\
		& 1331 to 9261 & 5.0630$\times 10^{-3}$  & 1.99986027 & 4.8790$\times 10^{-3}$ & 1.99986704 \\
		\bottomrule
	\end{tabular}
	\label{tab.5.1}
\end{table}
The numerical results confirm that the proposed finite element scheme achieves the optimal order of convergence when piecewise linear finite elements are employed.

\subsection{Long-time dynamics and validation}
In this subsection, we validate the proposed three-dimensional FEM scheme by comparing the long-time dynamics of the fish--mussel chemotaxis system with the numerical solutions of the corresponding spatially homogeneous ODE system obtained using MATLAB's built-in \texttt{ode45} solver. Let $(u,v,w)$ be a classical solution of system \eqref{1.1}-\eqref{1.3} satisfying \eqref{1.4}. The equilibrium point $(u^e,v^e,w^e)$ of system \eqref{1.1} are determined by the following system of equations
\begin{align}
	\left\{
	\begin{array}{r}
		-u^e(\alpha+\beta v^e+\gamma w^e)+f =\frac{\mbox{d}u^e}{\mbox{d}t}= 0,\\[2mm]
		\sigma_1v^e(1-v^e+u^e)-\delta_1 v^e =\frac{\mbox{d}v^e}{\mbox{d}t}= 0,\\[2mm]
		\sigma_2 u^e w^e-\delta_2 w^e =\frac{\mbox{d}w^e}{\mbox{d}t}= 0.
	\end{array}
	\right.\label{5.1}
\end{align}
The coexistence equilibrium corresponds to a balanced ecological state in which all interacting species persist simultaneously. The coexistence steady state $(\ust,\vst,\wst)$ is given by 
\begin{align*}
\ust=\frac{\delta _2}{\sigma _2},\: 
\vst=\frac{\delta _2 \sigma _1-\delta _1 \sigma _2+\sigma _2 \sigma _1}{\sigma _1 \sigma _2},\:
\wst=\frac{f \sigma _1 \sigma _2^2+\beta  \delta _1 \delta _2 \sigma _2-\alpha  \delta _2 \sigma _1 \sigma _2-\beta  \delta _2^2 \sigma _1-\beta  \delta _2 \sigma _1 \sigma _2}{\gamma  \delta _2 \sigma _1 \sigma _2}.
\end{align*}
For the coexistence equilibrium $(\ust,\vst,\wst)$ to be biologically feasible, that is, for all components to remain positive, the following conditions must hold $\delta_1\sigma_2 < \delta_2\sigma_1 +\sigma_1\sigma_2, \:\alpha+\beta+\frac{\beta\delta_2}{\sigma_2}
	<
	\frac{\beta\delta_1}{\sigma_1}
	+
	\frac{f\sigma_2}{\delta_2}$.
For the numerical simulations, we consider the three-dimensional domain
$\Omega=[-0.5,0.5]^3$, discretized using a uniform mesh with $\Delta x_i=\frac{1}{20}, i=1,2,3,$ and time step size $\Delta t=0.01$. The spatially homogeneous initial conditions are chosen as
\begin{align*}
	u(\cdot,0)=0,
	\qquad
	v(\cdot,0)=15,
	\qquad
	w(\cdot,0)=5.
\end{align*}
Moreover, the parameter values are taken as $\chi=1, \alpha=0.1, \beta=0.1, \gamma=0.1, \sigma_1=0.5, \sigma_2=0.1, \delta_1=0.1, \delta_2=1, f=15$. The numerical simulations indicate that the computed solution $(u,v,w)$ approaches the coexistence steady state $(10.0425, 10.8436, 3.0985)$ as $t\to\infty$, demonstrating the capability of the proposed numerical scheme to capture the long-time dynamics of the system. 

Figure \ref{fig.5.2} illustrates the spatial distributions of the solution components $u(\cdot,t), v(\cdot,t)$, and $w(\cdot,t)$ at $t=4$, obtained using the proposed three-dimensional finite element method. It can be observed that all solution components become nearly spatially homogeneous over the computational domain indicating that the transient dynamics gradually stabilize toward the coexistence equilibrium state. Figure~\ref{fig.5.3} demonstrates the asymptotic behaviour of the numerical solution $(u,v,w)$ obtained by both the MATLAB \texttt{ode45} solver and the proposed three-dimensional FEM. The time evolution profiles show that the nutrient density $u$, fish population $v$, and mussels population $w$ gradually stabilize and converge toward the coexistence steady state as $t\to 4$.
\begin{figure}[H]
	\centering
	\subfloat[$u(\cdot, t)$]{\includegraphics[width=0.33\textwidth]{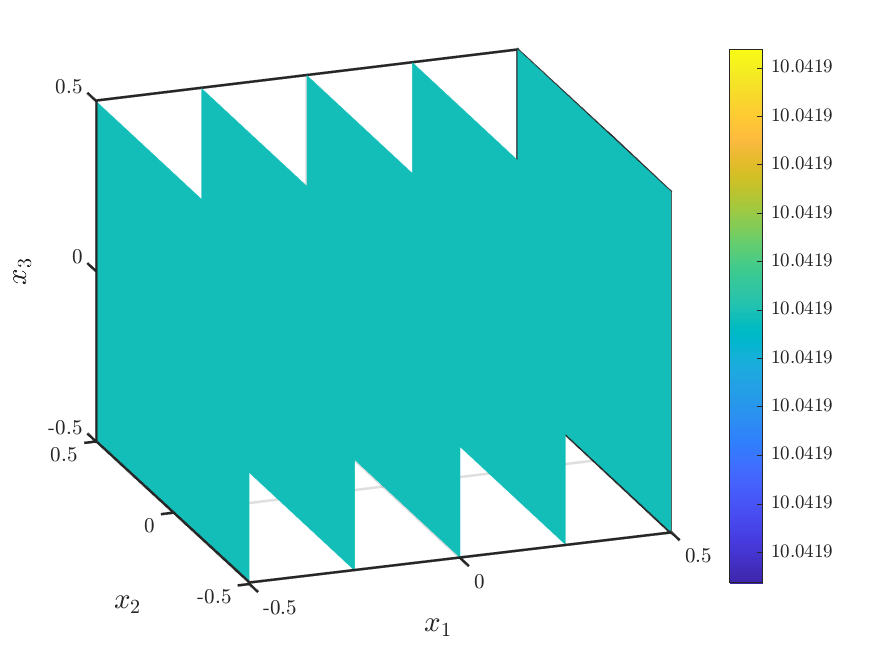}}
	\subfloat[$v(\cdot, t)$]{\includegraphics[width=0.33\textwidth]{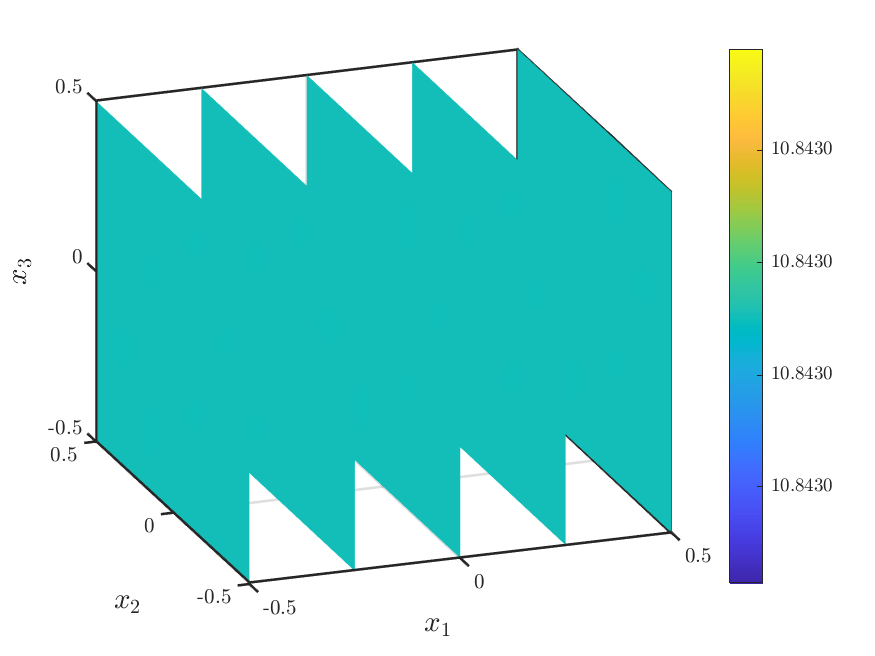}}
	\subfloat[$w(\cdot, t)$]{\includegraphics[width=0.33\textwidth]{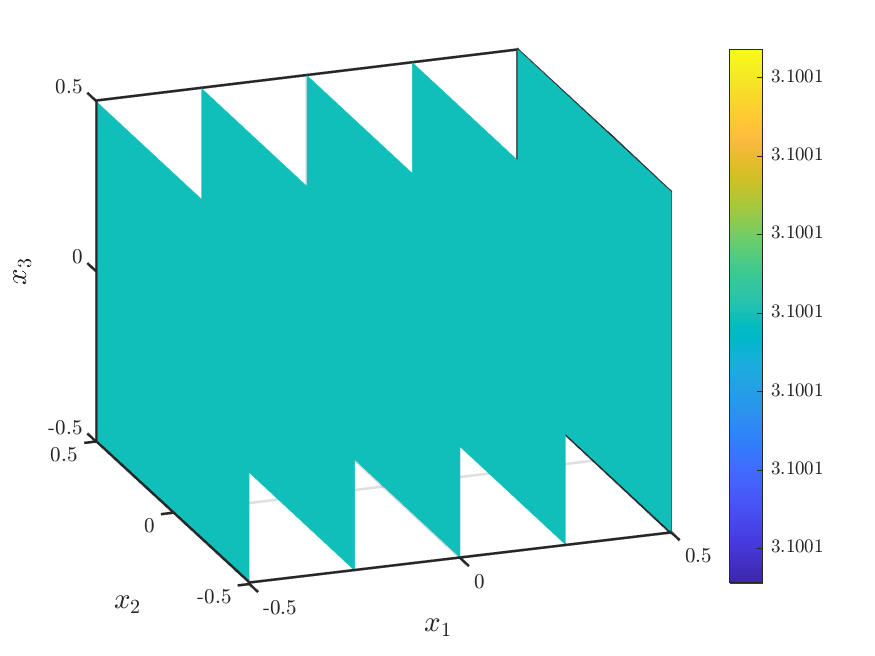}}
	\caption{Densities of the unknowns at $t=4$ by FEM.}
	\label{fig.5.2}
%\end{figure}
%\begin{figure}[H]
\centering
\subfloat[\texttt{ode45}]{\includegraphics[width=0.4\textwidth]{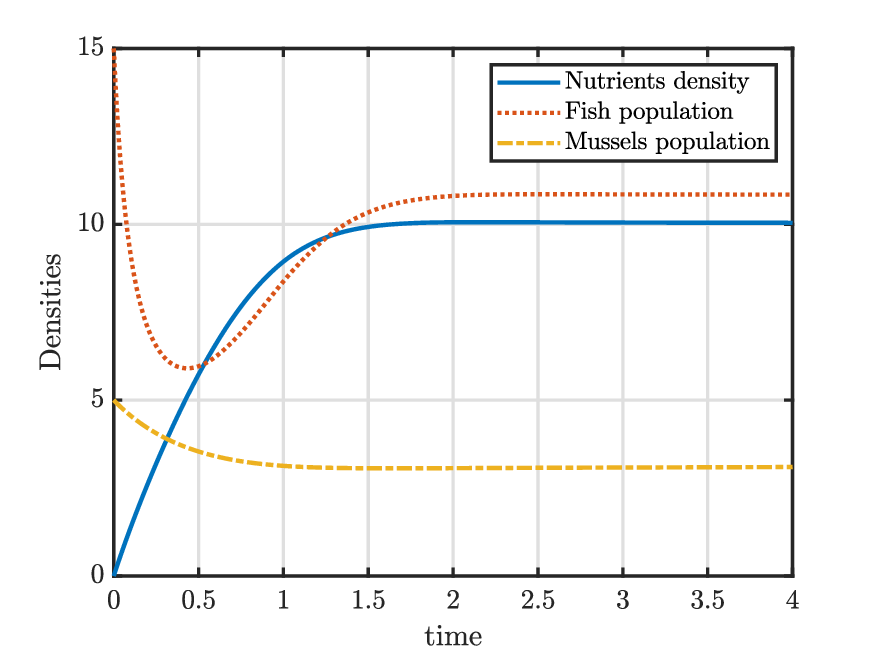}}
\subfloat[ FEM]{\includegraphics[width=0.4\textwidth]{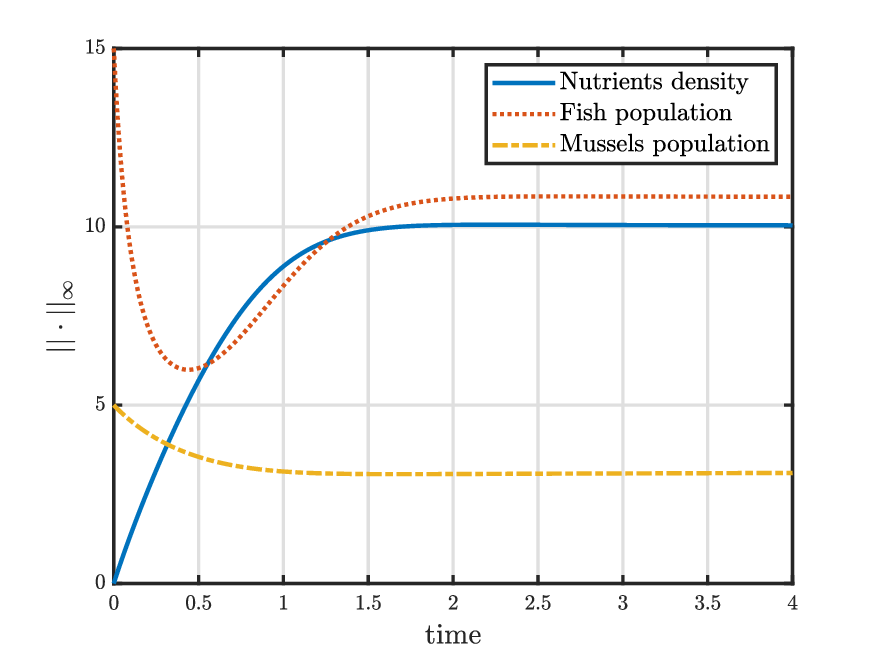}}
\caption{Comparison through asymptotic behaviour as $t\to 4$, where (A) Solution of \eqref{5.1} by \texttt{ode45}, (B) Solution of \eqref{1.1} by FEM.}
\label{fig.5.3}
\end{figure}
Table \ref{tab.5.4} presents the comparison between the solutions obtained using the MATLAB \texttt{ode45} solver and the FEM for considered system at different time levels. The numerical results demonstrate that the computed solutions gradually approach the coexistence steady state as time increases. In particular, the population densities stabilize near the equilibrium values confirming the long-time stability of the system. Moreover, the close agreement between the \texttt{ode45} and FEM results validates the accuracy and reliability of the proposed finite element numerical scheme in capturing the dynamics of the model \eqref{fig.5.4}.

\begin{table}[H]
	\centering
	\caption{Numerical validation table}
	\label{tab.5.4}
	\begin{tabular}{c ccc ccc ccc}
		\toprule
		& \multicolumn{3}{c}{\bf Solution of \eqref{5.1} by \texttt{ode45}} & \multicolumn{3}{c}{\bf Solution of \eqref{1.1} by FEM} 
		& \multicolumn{3}{c}{\bf Absolute error} \\
		\cmidrule(lr){2-4} \cmidrule(lr){5-7} \cmidrule(lr){8-10}
		\multicolumn{1}{c}{time} 
		& $u$ & $v$ & $w$ 
		& $u$ & $v$ & $w$
		& $u$ & $v$ & $w$ \\
		\midrule
		0.00 & 0.0000  & 15.0000 & 5.0000 & 0.0000  & 15.0000 & 5.0000 & 0 & 0 & 0\\
		1.00 & 8.9396  & 8.3778  & 3.1319 & 8.8917  & 8.3469  & 3.1397 & 0.0479 & 0.0309 & 0.0078\\
		2.00 & 10.0591 & 10.8067 & 3.0669 & 10.0739 & 10.7924 & 3.0691 & 0.0148 & 0.0143 & 0.0022\\
		3.00 & 10.0499 & 10.8520 & 3.0844 & 10.0493 & 10.8513 & 3.0861 & 0.0006 & 0.0007 & 0.0017\\
		4.00 & 10.0425 & 10.8436 & 3.0985 & 10.0418 & 10.8430 & 3.1000 & 0.0007 & 0.0006 & 0.0015\\
		\bottomrule
	\end{tabular}
\end{table}
\begin{figure}[H]
\includegraphics[width=0.45\textwidth]{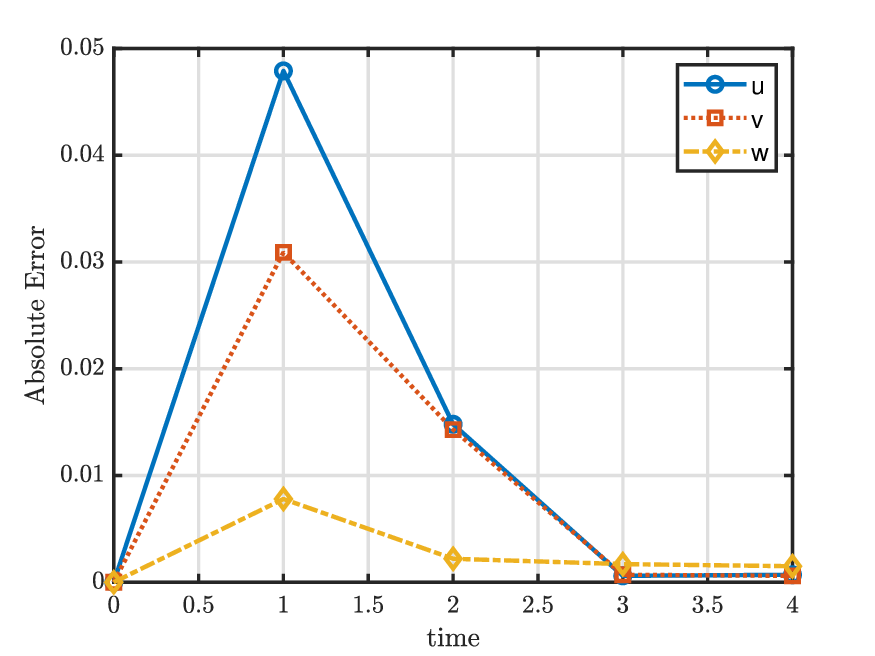}
\caption{Absolute error between \texttt{ode45} and FEM}
\label{fig.5.4}
\end{figure}

\subsection{Blow-up in a cubic domain}
Certain numerical observations indicate the possible occurrence of blow-up behaviour within a cubic domain. Although these results are preliminary in nature, they provide numerical evidence supporting the blow-up time estimate established in Theorem~\ref{t1.3}. The computations are carried out on the three-dimensional cubic domain
$[-0.5,0.5]^3,$ using a uniform spatial discretization with mesh size
$\Delta x_i=\frac{1}{20}, i=1,2,3.$ We consider the initial-boundary value problem associated with the Fish--Mussels chemotaxis system \eqref{1.1}, subject to the initial conditions
\begin{align}\label{5.2}
u(\cdot,0)=500e^{-500\sum_{i=1}^{3}x_i^2}, \quad v(\cdot,0)=1000e^{-1000\sum_{i=1}^{3}x_i^2}., \quad
w(\cdot,0)=800e^{-800\sum_{i=1}^{3}x_i^2}.
\end{align}
Moreover, the parameter values are chosen as $\tau=1, \chi=1, 
\alpha=\beta=\gamma=\sigma_1=\sigma_2=\delta_1=0.001, 
f=0.1$, and $\delta_2=0.5$.

The numerical approximations of the solution components $(u,v,w)$ are presented in Figure~\ref{fig.5.5}-\ref{fig.5.6}. The figure illustrates the temporal evolution of the solution profiles together with the corresponding evolution of their maximum values over the cubic domain $\Omega$. In particular, the numerical results reveal a rapid and uncontrolled growth of the component $v$ as $t$ approaches the critical time $T_{\max}=4\times10^{-5}$.
\begin{figure}[H]
	\centering
	\includegraphics[width=0.32\textwidth]{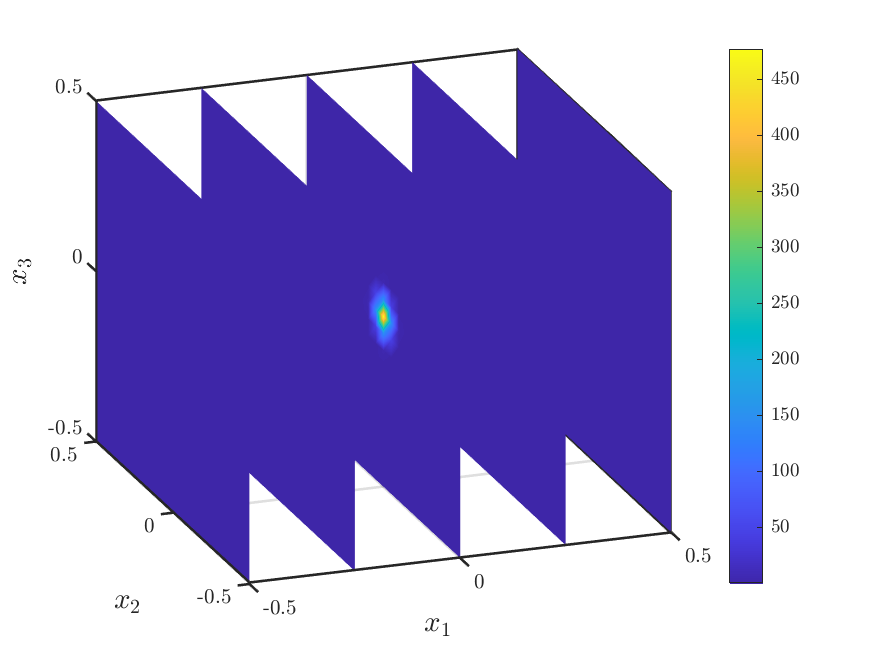}
	\includegraphics[width=0.32\textwidth]{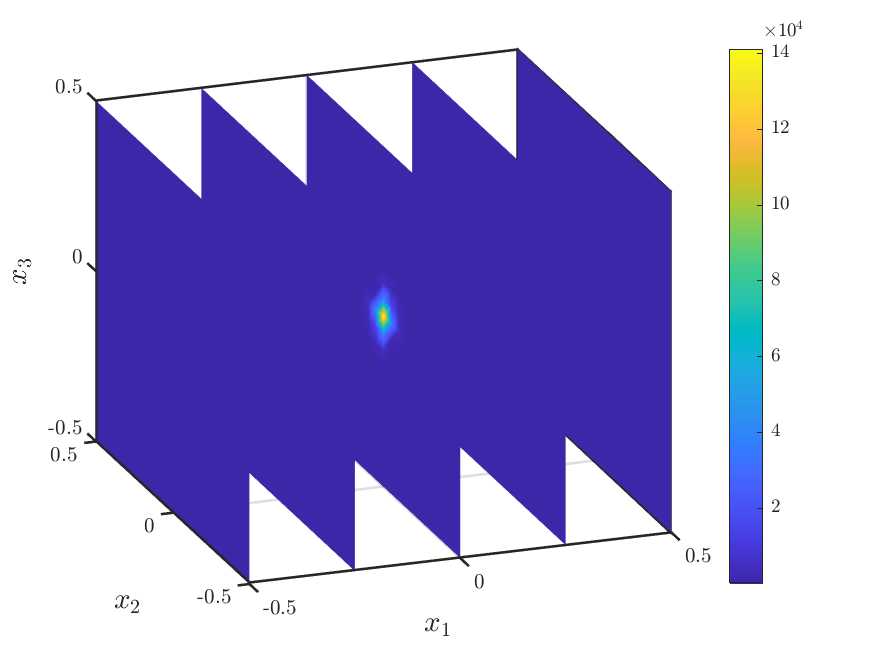}
	\includegraphics[width=0.32\textwidth]{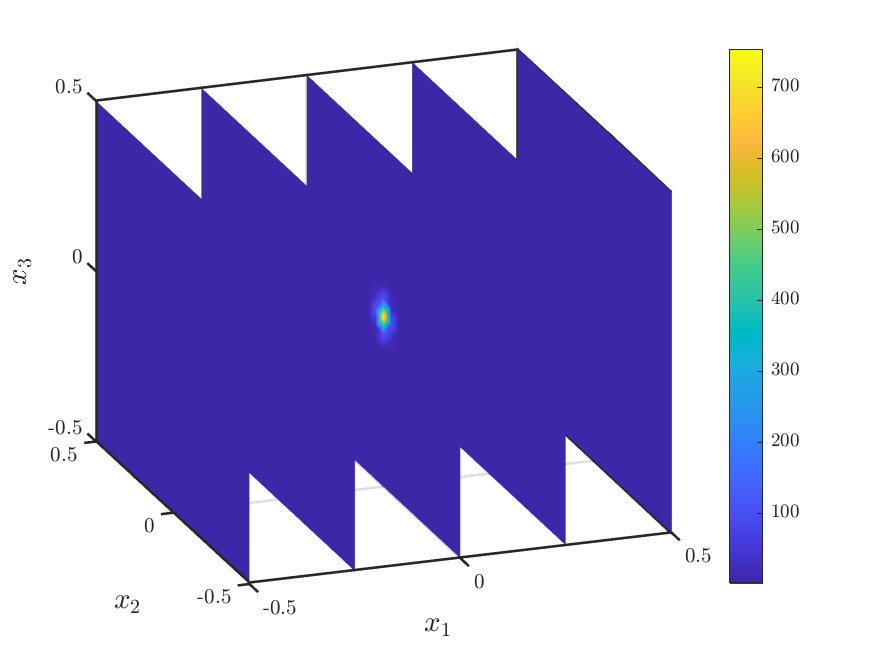}\\
	\subfloat[$u(\cdot,t)$]{\includegraphics[width=0.32\textwidth]{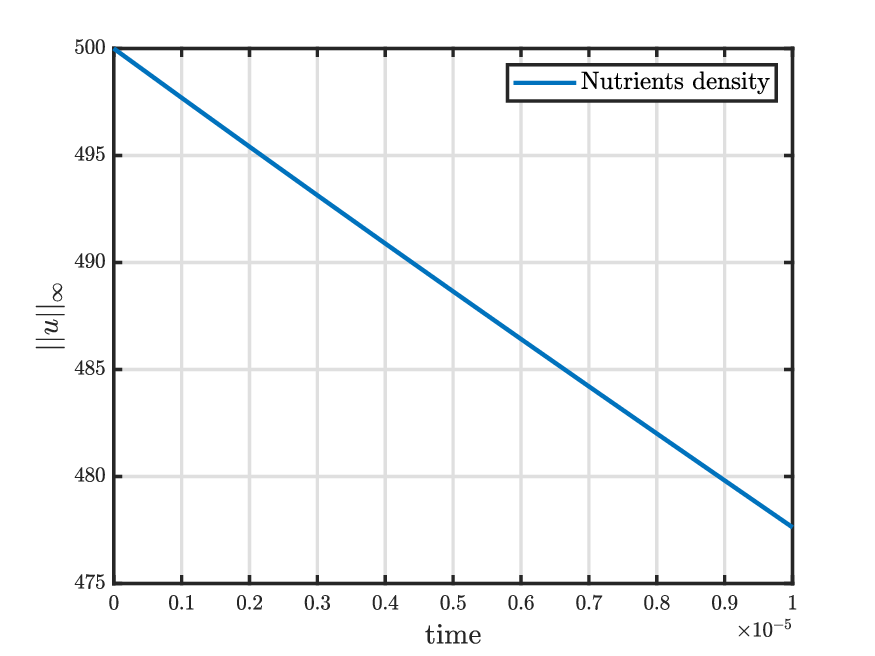}}
	\subfloat[ $v(\cdot,t)$]{\includegraphics[width=0.32\textwidth]{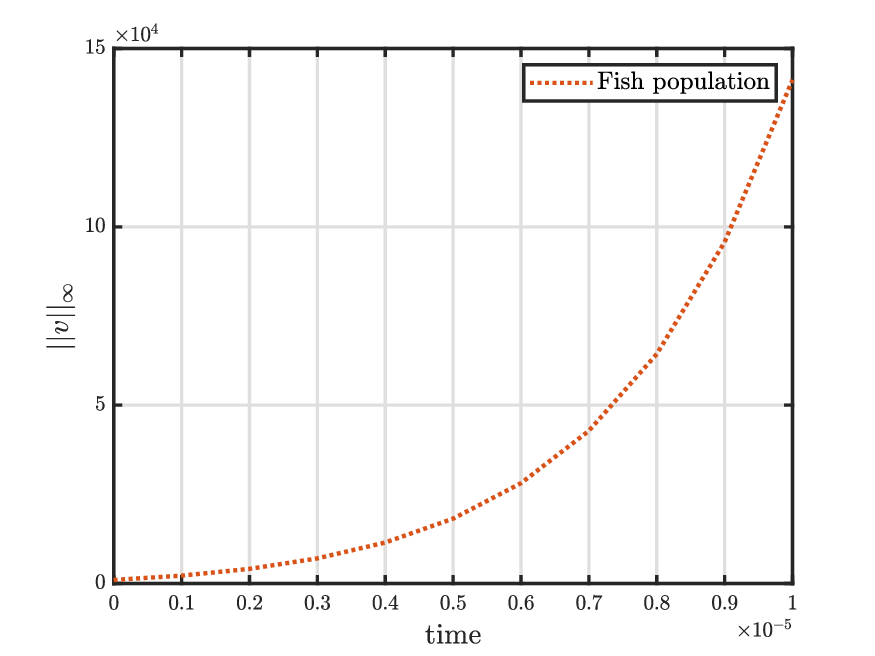}}
	\subfloat[ $w(\cdot,t)$]{\includegraphics[width=0.32\textwidth]{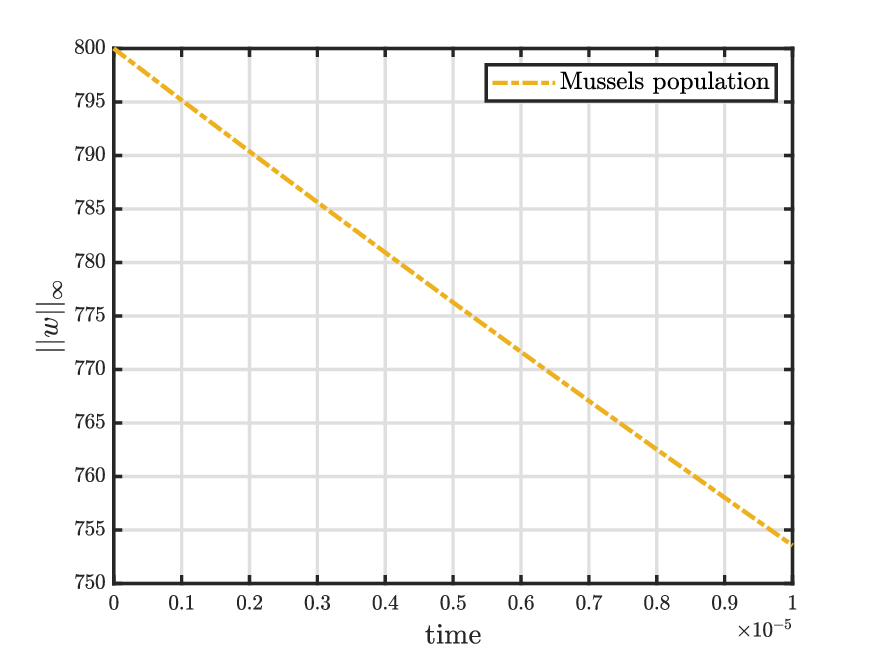}}
	\caption{Graphical representation of the evolution for unknowns at $t=10^{-5}$.}
	\label{fig.5.5}
%\end{figure}
%\begin{figure}[H]
	\centering
	\includegraphics[width=0.32\textwidth]{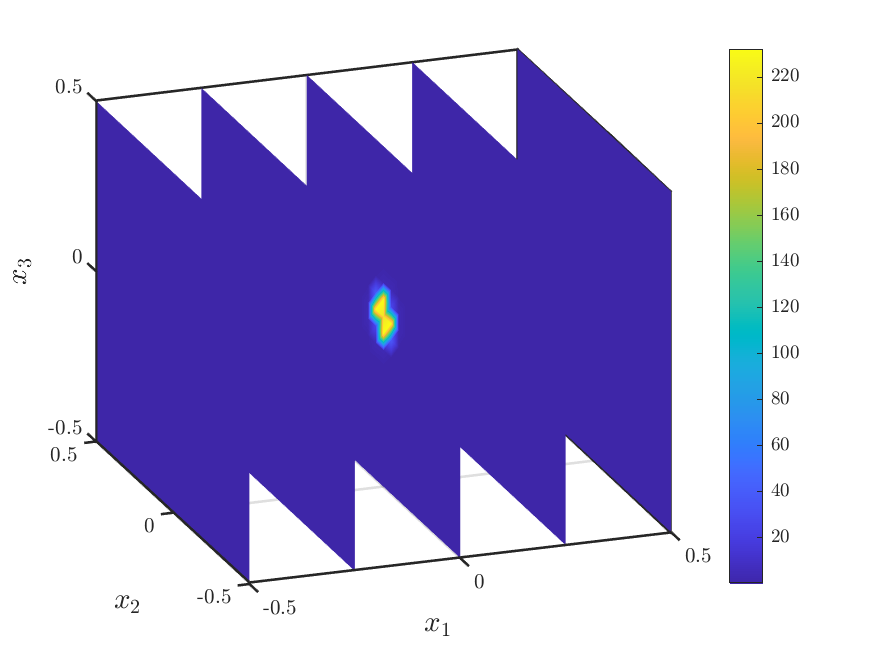}
	\includegraphics[width=0.32\textwidth]{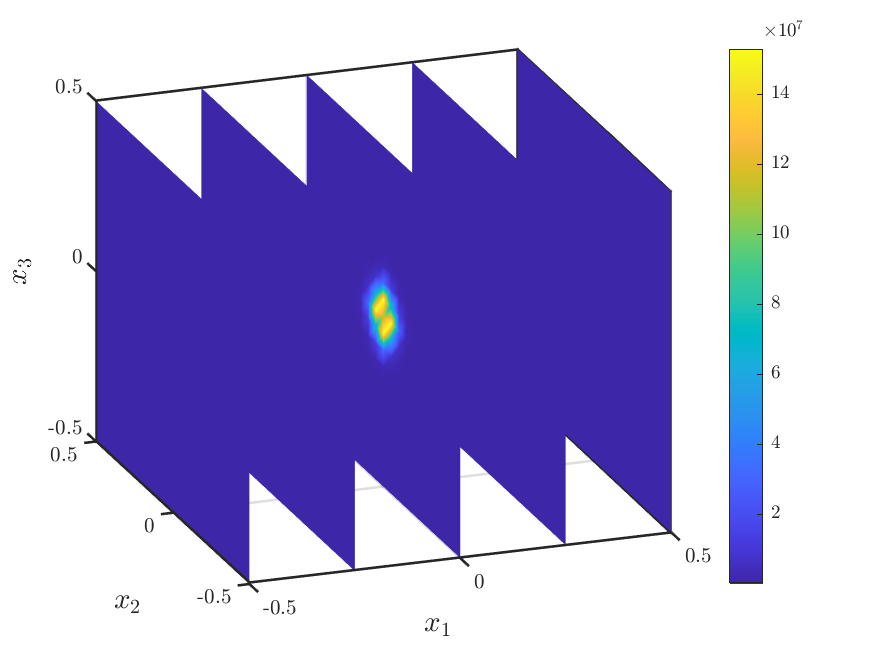}
	\includegraphics[width=0.32\textwidth]{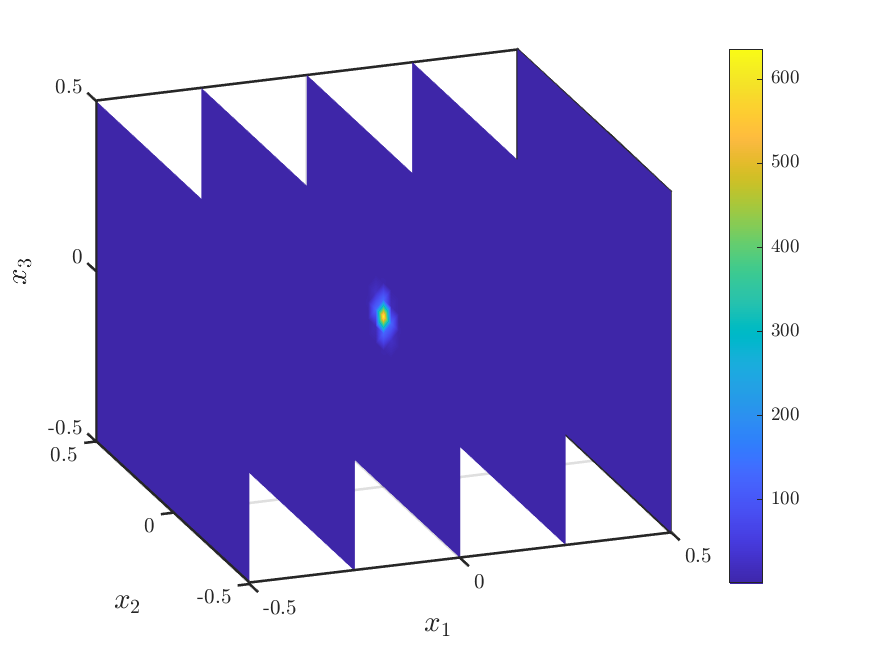}\\
	\subfloat[ $u(\cdot,t)$]{\includegraphics[width=0.32\textwidth]{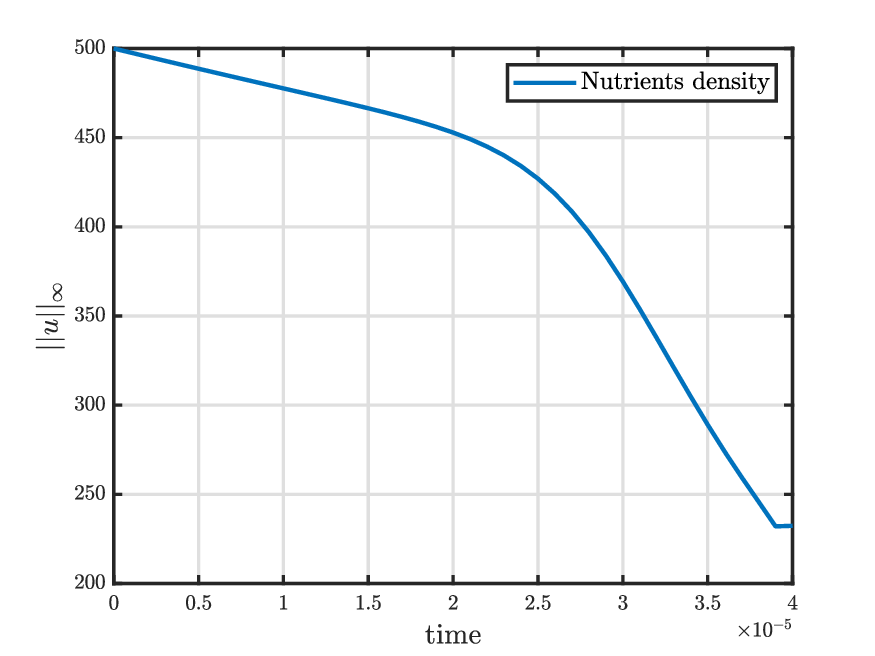}}
	\subfloat[ $v(\cdot,t)$]{\includegraphics[width=0.32\textwidth]{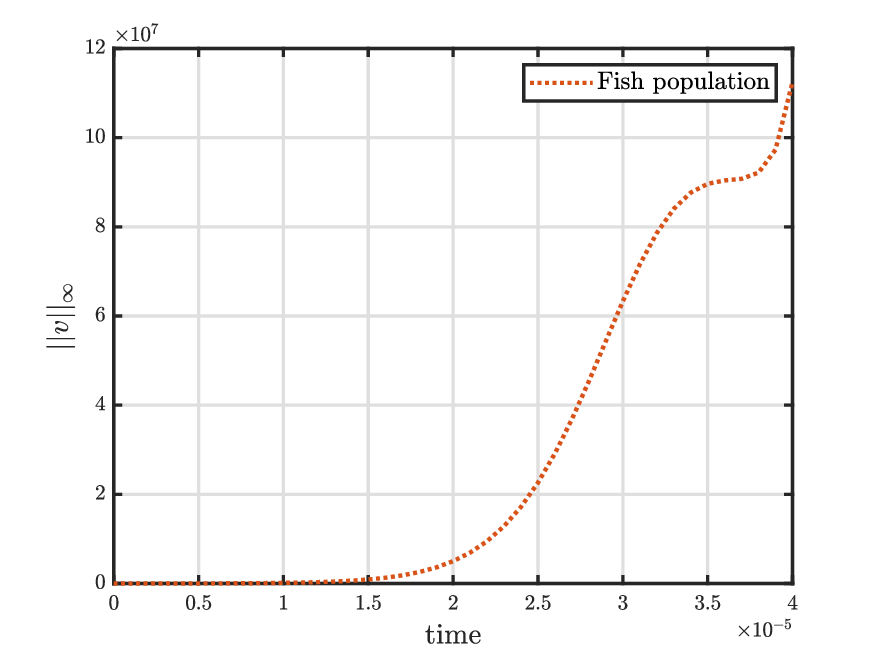}}
	\subfloat[ $w(\cdot,t)$]{\includegraphics[width=0.32\textwidth]{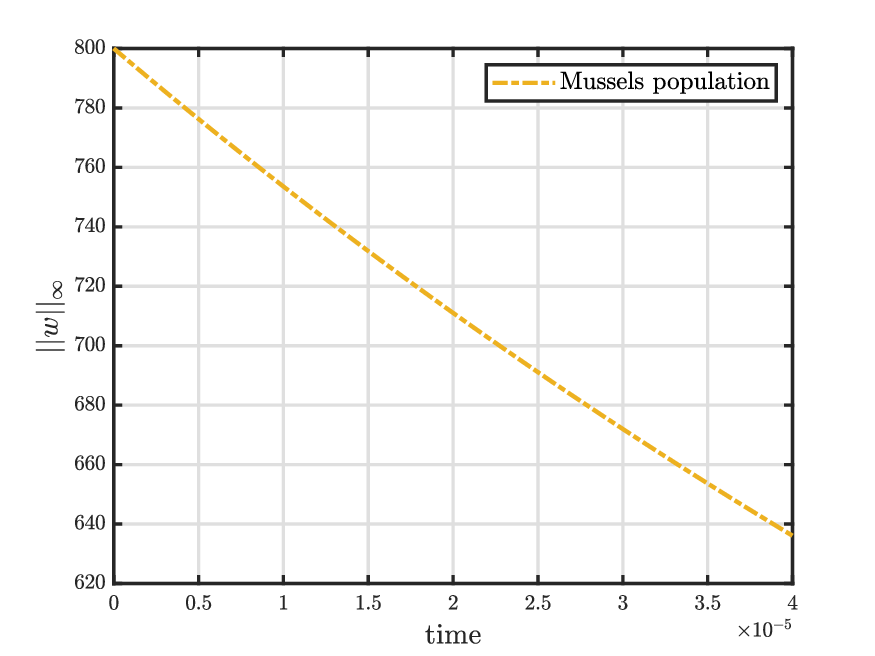}}
	\caption{Graphical representation of the evolution for unknowns at $t=4\times 10^{-5}$.}
	\label{fig.5.6}
\end{figure}
More precisely, the quantity $\|v(\cdot,t)\|_{\lis}$ attains a significantly large value near $t=T_{\max}$, as shown in Figure~ \ref{fig.5.6} B. Furthermore, in agreement with the theoretical estimates \eqref{l2.3.1} and \eqref{l2.3.2}, Figures~ \ref{fig.5.6} A and~C indicate that, for the parameter choice given in \eqref{5.2}, the corresponding upper bounds remain uniformly bounded by
\begin{align*}
K_1=\max\left\{	\|u_0\|_{\lis}, \frac{M}{\alpha}
	\right\}
	=500\leq \frac{\delta_2}{\sigma_2}, \quad 
K_2=\|w_0\|_{\lis}=800.
\end{align*}
These numerical observations provide additional evidence for the emergence of blow-up behaviour in the system and offer further insight into the underlying dynamics responsible for the rapid growth of the solution component $v$.

\section*{Declarations}
\subsubsection*{Funding}
GS and JS thank the Anusandhan National Research Foundation (ANRF), formerly Science and Engineering Research Board (SERB), Govt. of India for their support through ARG Matrics program (ARGM/2025/000442/MTR), during this work.

\subsubsection*{Author Contribution}
{\bf GS :} Writing – review \& editing, Writing – original draft, Validation, Methodology, Investigation, Conceptualization. {\bf JS :} Methodology, Review, Validation, Supervision,  Investigation.

\subsubsection*{Data Availability}
No data was used for the research described in the article.

\subsubsection*{Conflict of Interest}
The authors declare no conﬂict of interest.

\subsubsection*{Clinical Trial Number}
Not applicable.

\end{document}